\documentclass[12pt,a4paper]{amsart}
\usepackage[utf8]{inputenc}
\usepackage[english]{babel}
\usepackage{mathrsfs, amssymb, amsmath, amsfonts}
\usepackage{graphicx, tikz}
\usepackage{geometry, xcolor}
\usepackage{enumitem}
\usepackage{dsfont}
\usepackage{soul}
\allowdisplaybreaks[2]

\usepackage{subcaption}

\makeatletter
\renewcommand{\p@subfigure}{}
\makeatother

\title[$p$-Parabolicity on Graphs]{Characterizations of $p$-Parabolicity on Graphs}

\author{Andrea Adriani}
\address{Andrea Adriani, Vanguard Center, University Mohammed VI Polytechnic, 11103 Rabat/Salé, Morocco}
\email{andrea.adriani@um6p.ma}

\author{Florian Fischer}
\address{Florian Fischer, Institute for Applied Mathematics, University of Bonn, Endenicher Allee 60, 53115 Bonn, Germany}
\email{fischer@iam.uni-bonn.de}

\author{Alberto G. Setti}
\address{Alberto G. Setti, DiSAT, Universit\`a dell’Insubria, Via Valleggio 11, 22100 Como,
Italy}
\email{alberto.setti@uninsubria.it}

\usepackage[style=alphabetic,maxnames=5,minnames=2,abbreviate=false,date=long,url=false,isbn=false, doi=true,backend=bibtex, backref,giveninits=true, backrefstyle=none]{biblatex}
\usepackage{csquotes}
\AtBeginBibliography{\tiny}
\DefineBibliographyStrings{english}{%
	backrefpage = {\hspace{-0.4em}}, 
	backrefpages = {\hspace{-0.4em}} 
}

\usepackage[colorlinks=true,citecolor=blue,linkcolor=blue]{hyperref}

\newtheorem{theorem}{Theorem}[section]
\newtheorem{lemma}[theorem]{Lemma}
\newtheorem{proposition}[theorem]{Proposition}
\newtheorem{corollary}[theorem]{Corollary}

\theoremstyle{definition}
\newtheorem{example}[theorem]{Example}
\newtheorem{remark}[theorem]{Remark}
\newtheorem{definition}[theorem]{Definition}

\numberwithin{equation}{section}

\newcommand{\norm}[1]{\left\lVert #1 \right\rVert} 
\newcommand{\abs}[1]{\left\lvert #1\right\rvert} 
\newcommand{\set}[1]{\left\{ #1\right\} }
\newcommand\ip[2]{\langle #1, #2 \rangle}
\newcommand{\p}[1]{\left( #1 \right)^{\langle p-1 \rangle}}
\newcommand{\sse}{\subseteq}

\renewcommand{\epsilon}{\varepsilon}
\renewcommand{\phi}{\varphi}
\newcommand{\NN}{\mathbb{N}}
\newcommand{\ZZ}{\mathbb{Z}}
\newcommand{\TT}{\mathbb{T}}
\newcommand{\RR}{\mathbb{R}}

\newcommand{\dd}{\mathrm{d}}

\DeclareMathOperator{\sgn}{sgn}

\DeclareMathOperator{\cc}{cap}

\DeclareMathOperator{\Div}{div}
\newcommand{\FF}{F}
\newcommand{\DD}{D}
\newcommand{\E}{\mathcal{E}}

\addto\captionsenglish{
	
}

\newcommand{\Hmm}[1]{\leavevmode{\marginpar{\tiny%
			$\hbox to 0mm{\hspace*{-0.5mm}$\leftarrow$\hss}%
			\vcenter{\vrule depth 0.1mm height 0.1mm width \the\marginparwidth}%
			\hbox to 0mm{\hss$\rightarrow$\hspace*{-0.5mm}}$\\\relax\raggedright #1}}}
\begin{document}

\begin{abstract}
We study $p$-energy functionals on infinite locally summable graphs for $p\in (1,\infty)$ and show that many well-known characterizations for a parabolic space are also true in this discrete, non-local and non-linear setting. Among the characterizations are an Ahlfors-type, a Kelvin-Nevanlinna-Royden-type, a Khas'minski\u{\i}-type and a Poincar\'{e}-type characterization. We also illustrate  some applications and describe examples of graphs which are locally summable but not locally finite. {Finally,  we study the obstacle problem for the $p$-Laplacian using an approximation procedure by finite graphs in the summable, not necessarily locally finite,  case. This is then utilized to give an alternative proof of the  Khas'minski\u{\i}-type characterization. }
\\
	\\[4mm]
	\noindent  2020  \! {\em Mathematics  Subject  Classification.}
	Primary   39A12; Secondary  31C20; 31C45; 35R02.
	\\[4mm]
    \\[4mm]
	\noindent FF is supported by the DFG, and thanks Universit\`a dell’Insubria for the kind hospitality during several stays. Moreover, FF thanks the Heinrich Böll foundation for supporting one of the stays. AGS is member of the GNAMPA-INDAM group {``Equazioni Differenziali e Sistemi Dinamici''}. We would also like to thank an anonymous reviewer for valuable comments on an earlier version of this paper.  
\end{abstract}

\maketitle

\section{Introduction}

We consider an infinite weighted graph $G=(X,b,m)$, where $X$ is a countable infinite set of vertices, the edge measure $b$ is a symmetric locally summable function on $X\times X\to [0,\infty)$ which vanishes on the diagonal and $m$ is a measure of full support.

Analysis on infinite weighted graphs has undergone an impressive growth in recent years. 
Graphs can be viewed as discretizations of differentiable manifolds,  so that results obtained in the graph setting may shed light to what happens in the continuous setting and vice-versa.

However, the geometry of weighted graphs is rich enough to allow for a variety of different phenomena which sometimes do not have exact  counterparts even  on weighted Riemannian manifolds, in part due to the possible complete decoupling between vertex and edge measures.

Moreover, the natural differential operators  of differential geometry correspond to difference operators on graphs, which are intrinsically non-local and therefore some techniques which can be successfully used on manifolds do not carry over to graphs and alternative methods must be devised.

Most of this investigation involves the weighted discrete Laplacian $\Delta$, namely  the operator 
associated to the energy functional
\[
\mathcal{E} (f)=\frac 12 \sum_{x,y\in X} b(x,y)\abs{f(x)-f(y)}^2 ,
\]
which acts as 
\[
\Delta f(x)=\frac 1{m(x)}\sum_{y\in X} b(x,y)\bigl(f(x)-f(y)\bigr),
\]
whenever the expression makes sense.

Various generalizations to the case where $p\ne 2$ are possible. In this paper we consider a version of that studied by Nakamura and Yamasaki \cite{NY76} and by Yamasaki \cite{Y77}: 
for  $p\in(1,\infty)$,  we let  $\E_p (f)$ be the $p$-energy functional 
\[
\E_p (f)= \frac 12 \sum_{x,y\in X} b(x,y)\abs{f(x)-f(y)}^p,%
\]
which induces  the associated 
$p$-Laplacian, defined by
\[
\Delta_pf(x):=\frac{1}{m(x)}\, \sum_{y\in X} b(x,y)\abs{f(x)-f(y)}^{p-2}(f(x)-f(y))
\]
for all functions  $f$ such that the right hand side makes sense.

The $p$-Laplacian and its discrete counterparts have been studied extensively, both from the theoretical point of view - in particular, dealing with  spectral theory, eigenvalues estimates, Cheeger type inequalities, function theory and elliptic and parabolic equations -  and in view of applications to non-linear diffusion, image processing,  data clustering. Confining ourselves to the discrete setting we mention \cite{CC14,CZ21,GHJ21,HS21, HoS, HoS2,HuMu15,HuWa20,KM16,KimChung2010, Mu13,PC11b, Y77} for the former and  \cite{BH09, EDT17, ELB08, ETT15} for the latter.

In this paper we investigate  potential theoretic properties of the $p$-Laplacian, concentrating on parabolicity. As in the continuous case, there are deep connections and interplays between  properties of the $p$-energy functional, its associated $p$-Laplacian and the geometry of the underlying graph.

For $p=2$, the  parabolicity of the Laplacian can be defined and characterized in a numbers of equivalent ways: in terms of recurrence of the induced process, of  the constancy of bounded below super-harmonic functions,   of the capacity induced by the energy functional, and of the validity of Green's formulae or of maximum principles at infinity.

This  holds also true in the case of the $p$-Laplacian for $p\ne 2$ and, starting from the seminal paper by Yamasaki \cite{Y77}, several previous works were devoted to the investigation of various facets of $p$-parabolicity, see e.g. \cite{F:AAP,F:GSR, Y84, MPR22, MPR24a, MPR24,Soa2, Soa1, SY93Class, SY93Para, Prado}.

The original motivation for writing this paper was to show that $p$-parabolicity, which we define in terms of the vanishing of the variational $p$-capacity of finite sets, is equivalent to the existence of a Khas'minski\u{\i} potential, namely a function $\kappa$ that blows up at infinity and satisfies $\Delta_{p}\kappa \geq 0$ outside a finite set, in this non-linear, non-local and non-locally finite setting. While the sufficiency condition depends on the maximum principle at infinity,  for the other implication we construct a potential via variational methods. An alternative proof using solutions to the obstacle problem on graphs similar to the corresponding proof on manifolds in \cite{Valtorta} is given in the appendix.  {Solutions to the obstacle problem are obtained by variational methods, using an approximation procedure by finite graphs in the summable, not necessarily locally finite,  case.}

Along the way, we realized that many well-known characterizations of $p$-parabolicity were not written down in this general setting. So we decided to  provide a comprehensive treatment of the  parabolicity of the $p$-Laplacian on  weighted non-locally finite graphs. This also gave us the possibility to prove an area-type sufficient  condition for $p$-parabolicity which is a characterization on model graphs and seems to be new for $p\neq 2$.

We show these characterizations, on the one hand because we need some of these results for the proof of the  Khas'minski\u{\i} criterion, and also to provide a good reference for applications.

The paper is organized as follows. In Section~\ref{sec:setting} we introduce the notation and the main definitions, and collect some functional analytic results that will be used later in the paper. In particular, we define the $p$-energy functional and the associated  $p$-Laplacian and $p$-capacity. We then define $p$-parabolicity in terms of the vanishing of the $p$-capacity on all finite subsets and recall that it can be characterized  by the existence of so called null sequences, namely, sequences of functions pointwise converging to $1$ with $p$-energy tending to $0$.   Section~\ref{s:chara} is the heart of the paper and describes a number of different characterizations of $p$-parabolicity. We begin with characterizations  which are well-known in the locally finite case:  in terms of the existence of a $p$-Green's function (Section~\ref{s:Green}), in terms of function theoretic properties of the space of function of finite $p$-energy (Section~\ref{s:func}), and of Liouville properties  for non-negative $p$-superharmonic functions (Section~\ref{s:super}). Sections~\ref{sec:KNR} to \ref{sec:K} describe some of our more original results. In Section~\ref{sec:KNR} we describe the extension of a Kelvin-Nevanlinna-Royden characterization to the non-locally finite case, { and, in Section~\ref{s:Poincare}, we provide a characterization in terms of the validity of a Poincar{\'e} type-inequality}. In Section~\ref{sec:A} we show a ``maximum principle at infinity''-type characterization which goes back to Alhfors, \cite{Ahlfors}, for  Riemann surfaces and for $p=2$, and is new in  the case of graphs. 
In Section~\ref{sec:LPS} we generalize to all $p>1$ a characterization in terms of  the vanishing of a suitably defined capacity of subsets of the boundary at infinity which was recently obtained in the linear case in \cite{LPS23}. Section~\ref{s:area} describes a sufficient condition for the $p$-parabolicity of locally finite graphs which is the exact counterpart of an area condition valid on Riemannian manifolds and which becomes a characterization in the case of model graphs.  Section~\ref{sec:WMP} gives a  weak maximum principle at infinity characterization of $p$-parabolicity. There we also show that the existence of a Khas'minski\u{\i} potential implies $p$-parabolicity. Finally, in Section~\ref{sec:K} we show that the Khas'minski\u{\i} condition introduced in the previous section is, in fact, also necessary. In the appendix we show an alternative proof adapting the idea given in \cite{Valtorta}.

The final section describes some applications of the results obtained in several examples, focusing on locally summable but not locally finite graphs.

\section{Setting the Scene}\label{sec:setting}

\subsection{Graphs and Laplace Operators}
A \emph{(weighted) graph} $G$  is a triple $G=(X,b,m)$, where
\begin{itemize}
\item $X$ is a countable infinite set;
\item $ b$ is a symmetric function  $b\colon X\times X \to [0,\infty)$ with zero diagonal which is locally summable,  i.e., the vertex degree function satisfies \[ \deg(x):=\sum_{y\in X}b (x,y)<\infty, \qquad x\in X;\]
\item $ m \colon X \to (0,\infty) $ is a measure of full support.
\end{itemize}

 The elements of $X$ are called \emph{vertices} and the function $b$ represents weights of edges between vertices. Two vertices $x, y$ are said to be \emph{connected}, and we write $ x \sim y $, if $b(x,y)>0$. A set $V \sse X$ is called \emph{connected} if for every two vertices $x,y\in V$ there are vertices ${x_0,\ldots ,x_n \in V}$, such that $x=x_0$, $y=x_n$ and $x_{i-1}\sim x_i$ for all $i\in\set{1,\ldots, n-1}$. Throughout this paper we will always assume that
 $X$ is connected. The set of edges $E_{V}=E_{V,b}$ on $V$ are given by
 \[E_V:=\set{(x,y)\in V\times V: x\sim y}.\]
 We set $E=E_X$.

 For $V\sse X$, let $\partial_{e} V=\set{y\in X\setminus V : y\sim z\in V}$ and $\partial_{i} V=\set{y\in  V : y\sim z\in X\setminus V}$ denote the \emph{exterior} and \emph{interior boundary} of $V$, respectively. The sets  $\mathring{V}=V\setminus \partial_i V$  $\overline{V}:=V\cup \partial_e V$  are the \emph{interior} and \emph{closure} of $V$ (these are obviously not meant in topological sense).

A graph $G$ is called \emph{locally finite} if for all $x\in X$
\[\# \set{y\in X : y\sim x  }< \infty.\]

A function $f\colon \RR\to \RR$ is \emph{odd} if $f(t)=-f(-t)$, $t\in \RR$.

 The space of real valued functions on $V\subseteq X$ is denoted by $C(V)$ and the space of functions with compact (= finite) support in $V$ is denoted by $ C_c(V)$. We consider $C(V)$ to be a subspace of $C(X)$ by extending the functions of $C(V)$ by zero on $X\setminus V$.

For all $1\leq p < \infty$, $V\sse X$, we define 
\begin{align*}
\ell^p(V,m)&:=\bigl\{f\in C(X): \norm{f}_{p,m, V}^p:=\sum_{x\in V}\abs{f(x)}^pm(x)<\infty\bigr\},\quad \text{and} \\
\ell^{\infty}(V)&:=\bigl\{f\in C(X): \sup_{x\in V}\abs{f(x)}<\infty\bigr\}.
\end{align*}
Note that $(\ell^p(X,m),\norm{\cdot}_{p,m})$ is a reflexive Banach space for $p\in (1,\infty)$, and a Banach space for  $p=1,\infty$. {The $\ell^p$-spaces on $X\times X$ are defined in a similar way.}

 Let $f,g\in C(X)$, then $f\vee g$, $f\wedge g$, $f_-$ and $f_+$ denote {the pointwise maximum and  minimum of $f$ and $g$, and  the negative  and positive part of $f$, respectively}. By $\mathds{1}_V$ we denote the characteristic function on $V\sse X$, and for singletons $\set{x}\sse V$, we set $\mathds{1}_x:=\mathds{1}_{\set{x}}$, and identify $1=\mathds{1}_X$ when no misinterpretation can occur.

We also define the linear difference operator $\nabla { \colon C(X)\to C(X\times X)}$ by the formula
\[\nabla_{x,y}f=f(x)-f(y), \qquad x,y\in X, f\in C(X).\]

We now introduce {Laplace}-type operators (which are sometimes also called Schrödinger operators): For $p\in [1,\infty)$ and  $V\sse X$, the \emph{formal space} $ \FF^p(V)=\FF^p_{b}(V) $ is defined by
\[
\FF^p(V):= \{ f\in C(X): \sum_{y\in X} b(x,y)\abs{\nabla_{x,y}f}^{p-1} < \infty  \mbox{ for all } x\in V  \},\\
\]
and we write $\FF^p=F^p(X)$. Note that on locally finite graphs, we have $\FF^p=C(X)$ for all $p\in [1,\infty)$.

Moreover, the local summability condition implies that $f\in \FF^p(V)$ if and only if, for every $x\in V$,
\[
\sum_{y\in X} b(x,y)|f(y)|^{p-1}<\infty.
\]
As a consequence, for all $p\geq 1$,
\[
\ell^\infty(X) \sse \FF^p(X).
\]

For all $p\geq 1$ and $a\in \RR$, we set
\[ \p{a}:= |a|^{p-1} \sgn (a)=|a|^{p-2} a, \]
with the usual {convention} that $\infty\cdot 0 =0$. Here, $\sgn\colon \RR\to \set{-1,0,1}$ is the sign function, that is $\sgn(\alpha)=1$ for all $\alpha> 0$, $\sgn(\alpha)=-1$ for all $\alpha< 0$, and $\sgn(0)=0$.

Let $p\in [1,\infty)$, then the \emph{($p$-)Laplacian}  $\Delta_{p}=\Delta_{p,b, m} \colon \FF^p(V)\to C(V)$ is  defined via
\[ \Delta_pf(x):=\frac{1}{m(x)}\, \sum_{y\in X} b(x,y)\p{\nabla_{x,y}f}, \qquad x\in V.\]

A function $u\in \FF^p(V)$ is said to be \emph{($p$-)harmonic (respectively ($p$-)superharmonic, strictly ($p$-)superharmonic, ($p$-)subharmonic}) on $V\sse X$ with respect to $\Delta_p$ if \[\Delta_pu=0 \quad(\text{respectively }\Delta_pu\ge 0,\,\Delta_pu \gneq 0, \, \Delta_pu\leq 0)\qquad\text{ on }V.\]
If $V=X$ we only speak of super-/sub-/harmonic functions.

\subsection{Energy Functionals Associated with Graphs}
For $p\in [1,\infty)$ and $V\sse X$, let the \emph{space of finite energy functions} $\DD^{p}(V)=\DD^p_{b,c}(V)$ be given by
\begin{align*}
\DD^{p}(V)&:=\bigl\{f\in C(X): \norm{\nabla f}^p_{p,b,V}:=\sum_{x,y\in V} b(x,y)\abs{\nabla_{x,y}f}^p<\infty \bigr\},
\end{align*}
and let  $\DD^p:=\DD^p(X)$ and $\norm{\nabla \cdot}_{p,b}=\norm{\nabla\cdot}_{p,b,X}$.

We define the \emph{(p-)energy functional} $\E_p=\E_{p,b,c}\colon \DD^{p}\to [0,\infty)$ via
\begin{align*}
	\E_p(f):=\frac{1}{2}\sum_{x,y\in X} b(x,y)\abs{\nabla_{x,y}f}^p=\frac{1}{2}\norm{\nabla f }_{p,b}^p.
\end{align*}

For $p=2$ such a functional is a quadratic form.

Note that $\E_p$ is \emph{Markovian} for all $p\in[1,\infty)$, i.e., compatible with normal contractions. This means that $\E_p(C\circ f)\leq \E_p(f)$ for all $f\in \DD^{p}$ and all $C\colon \RR\to \RR$ with $C(0)=0$ and $\abs{C(s)-C(t)}\leq \abs{s-t}$ for all $s,t\in \RR$. Furthermore, we have the following fundamental property:

\begin{lemma}\label{lem:BeurlingDeny}
	Let $p\in [1,\infty)$, then $\E_p$ satisfies the second Beurling-Deny criterion, that is, for all $f,g \in \DD^p$ and all normal contractions $C\colon \RR \to \RR$, we have
	\[\E_p(f+C\circ g)+\E_p(f-C\circ g)\leq \E_p(f+g)+\E_p(f-g).  \]
	In particular,
	\[ \E_p(f\wedge g)+\E_p(f\vee g)\leq \E_p(f)+\E_p(g).\]
\end{lemma}
\begin{proof}
By the $p$-triangle inequality and the assumption $f,g\in \DD^p$ all of the sums above converge. The statement follows if we can show the associated pointwise inequality and then sum over all $x,y\in X$. Hence, we need to show
\[\abs{\nabla_{x,y}f + \nabla_{x,y}C\circ g}^p+\abs{\nabla_{x,y}f - \nabla_{x,y}C\circ g}^p \leq \abs{\nabla_{x,y}f + \nabla_{x,y}g}^p+ \abs{\nabla_{x,y}f - \nabla_{x,y}g}^p  \]
for all $x,y\in X$. Let $\alpha:=\nabla_{x,y}f$, $s:=\nabla_{x,y} C\circ g$ and $t:=\nabla_{x,y}g$. Since $C$ is a normal contraction, $\abs{s}\leq \abs{t}$. Hence, the inequality rewrites to
\[ \abs{\alpha + s}^p+\abs{\alpha - s}^p \leq \abs{\alpha + t}^p+ \abs{\alpha - t}^p. \]
The function $F_{\alpha}(s):=\abs{\alpha + s}^p+\abs{\alpha - s}^p$ is even and convex for all $\alpha\in \RR$ and $p\geq 1$. Thus, $F_{\alpha}(s)\leq F_{\alpha}(t)$ for all $\abs{s}\leq \abs{t}$, which is the desired inequality.

The second assertion follows by taking $\tilde{f}:= (f+g)/2$, $\tilde{g}:=(f-g)/2$ and $C=\abs{\cdot}$.
\end{proof}
For more information on nonlinear energy functionals satisfying the second Beurling-Deny criterion see \cite{BDS, CG, Claus, Puchert, SZ25}.

Note that an application of H\"older's inequality shows that
\[
\DD^{p}(V)\sse \FF^p(V), \quad  p\in [1,\infty).
\]
The connection between $\Delta_p$ and $\E_p$ on $C_c(X)$ for $p\in[1,\infty)$ is further described via a Green's formula, also known as integration by parts, which is stated next.

Let $V\sse X$. To shorten notation, we define a weighted bracket $\ip{\cdot}{\cdot}_{V}$ via
\[\ip{f}{\phi}_{V}:=\sum_{x\in V}f(x)\phi(x)m(x)\]
whenever the sum converges absolutely for $f, \phi \in C(X)$. 
Moreover,  for any $u,v\in C(V)$ and $p\in [1,\infty)$ we set
\[
\E_{p,{V}}(u,v):=\frac{1}{2} \sum_{x,y \in {V}}b(x,y)\p{\nabla_{x,y}u}\nabla_{x,y}v,
\]
whenever the sums converge absolutely; and if $V=X$, we set $\E_{p}=\E_{p,{V}}$.

\begin{lemma} \label{lem:GreensFormula} 
	Let $p\in [1,\infty)$ and $V\sse X$. Let $f\in \FF^p(V)$ and $\phi\in C_c(X)$. Then all of the following sums converge absolutely and
	\begin{align*}
		\ip{\Delta_pf}{\phi}_{V}=\E_{p,V}(f,\phi)+\sum_{x\in V, y\in { \partial_e V}}b(x,y)\p{\nabla_{x,y}f}\phi(x).
	\end{align*}
	In particular, the formula can be applied to $f\in C_c(X)$, or $f\in \DD^{p}$, and
	\[\E_p(\phi)=\ip{\Delta_p\phi}{\phi}_{V}, \qquad \phi\in C_c(V).\]
\end{lemma}
The  proof follows from  direct computations, see, e.g.,  \cite[Lemma~2.3]{F:GSR}, where a version valid for arbitrary $p$-Schrödinger operators is described.
\begin{remark}
	In the following, we will always assume that 
	\[ p\in (1,\infty), \]
	even though some (but not all) of the results are also valid for  $p=1$. {This will be pointed out, whenever it is the case.} 
	
	Moreover, some of the statements also remain valid if one adds a non-negative potential to the $p$-Laplacian, i.e., considers an operator $L_p$ instead of $\Delta_p$ which acts via $L_pf(x)=\Delta_pf(x)+ V(x)\p{f(x)}$ for $f\in\FF^p$, $x\in X$ and some $0\leq V\in C(X)$. In some instances it is enough to assume that the associated $p$-energy functional is non-negative on $C_c(X)$. {We will also comment on these issues when appropriate.}
\end{remark}

\begin{example}
	Our setting includes the \emph{fractional $p$-Laplacian} $\Delta_p^{\sigma}$, $\sigma\in (0,2/p)$, over a countable infinite set $X$, see e.~g. \cite{DKP26}. The operator $\Delta_p^{\sigma}$ is defined via 
	\[ \Delta_p^\sigma f(x):= C\cdot\int_0^\infty e^{-t\Delta_2}(\p{\nabla_{x \cdot}f})(x) \frac{\dd t}{t^{1+\frac{p\sigma}{2}}} \]
	for $f\in C_c(X)$.	Here, $P_t=e^{-t\Delta_2}, t\geq 0$, is the semigroup associated with the Dirichlet Laplacian to $\Delta_2=\Delta_{2,b,m}$ and $C=C(p,\sigma)>0$ is a constant (a quotient of certain Gamma functions). For $p=2$, it is the classical fractional Laplacian, see e.~g. \cite{CR18,CRSRTV,KN23}. Let $p_t(x,y)= e^{-\Delta_2}1_y(x)$, $x,y\in X$, be the heat kernel to $\Delta_2$, and set  
	
	\[b_{p,\sigma}(x,y):= \int_0^\infty p_t(x,y)\frac{\dd t}{t^{1+\frac{p\sigma}{2}}},\]
	whenever $x\neq y$ as well as $b_{p,\sigma}(x,x)=0$.
	
	For all $f\in C_c(X)$, we have
	\begin{align*}
		\Delta_p^\sigma f(x)= C\int_0^\infty \sum_{y\in X} p_t(x,y) \p{\nabla_{x,y}f}\frac{\dd t}{t^{1+\frac{p\sigma}{2}}}=C\sum_{y\in X}b_{p,\sigma}(x,y) \p{\nabla_{x,y}f}.
	\end{align*} We extend $\Delta_p^\sigma$ to functions in $\FF^p_{b_{p,\sigma}}$ via the graph $p$-Laplacian description. Note that $(X,b_{p,\sigma},C^{-1})$ is a non-locally finite complete graph which is locally summable.
	
	Other examples can be found in Section~\ref{s:Appl}.
\end{example}

\begin{remark}
	It should be mentioned that another useful generalization of the Laplacian arises via a Green's formula	from the $p$-energy
	$$
	\sum_{x\in X} \frac 1{2}\left(\sum_{y\in X}b(x,y)(f(x)-f(y))^2\right)^{p/2},
	$$
	where the inner summation represents the square of the length of the ``gradient''  of $f $ at $x$, see, e.g. \cite{GLY16}. However, this version of the $p$-Laplacian will not be considered here.
\end{remark}

We continue with some basic observations about the space of finite energy functions. We define the space $D_{0}^{p}:=D^p_{0,b,c}(X)$ to be the closure of $C_c(X)$ with respect to the energy norm $\norm{\cdot}_{o,p}$ given by 
\begin{align*}
	\norm{f}_{o,p}:=(\E_p(f)+\abs{f(o)}^p)^{1/p},
\end{align*}
where $o$ is some fixed vertex. If $p=2$, this energy norm is called form norm. The first statement of the following lemma can be found e.g. in \cite[Lemma~2.7]{Mu13}.  {We provide a proof for the convenience of the reader}.

\begin{lemma}\label{lem:uniform}
For $p\in (1,\infty)$, and $o\in X$, the spaces $(D^{p}, \norm{\cdot }_{o,p})$ and $(D_{0}^{p}, \norm{\cdot }_{o,p})$ are uniformly convex and therefore reflexive. Moreover, $\norm{\cdot}_{o,p}$ and $\norm{\cdot}_{o',p}$ are equivalent norms on $D^{p}$ for all $o, o'\in X$. In particular, for all $x\in X$ the point evaluation $\operatorname{ev}_x\colon D^p\to \RR, f\mapsto f(x)$ is a continuous linear functional on $(D^{p}, \norm{\cdot }_{o,p})$.
\end{lemma}
\begin{proof}
	It is not difficult to see that $\norm{\cdot}_{o,p}$ is a norm on $D^p$. It follows from Fatou's lemma in the same fashion as in the linear case in \cite[Proposition~1.3]{KLW21} that $\E_p$ is lower semi-continuous, and thus, that $(D^{p}, \norm{\cdot }_{o,p})$ is complete. Since $(D_{0}^{p}, \norm{\cdot }_{o,p})$ is a closed subspace of a Banach space, it is also complete. Moreover, since both  $(D^{p}, \norm{\cdot }_{o,p})$ and $(D_{0}^{p}, \norm{\cdot }_{o,p})$ are (isometric to) closed subspaces of $L^p(\set{o}\times E, 1\times b)$, they are uniformly convex.

	It remains to show that the norms $\norm{\cdot}_{o,p}$ and $\norm{\cdot}_{o',p}$ are equivalent on $D^{p}$ for all $o, o'\in X$. Let us consider a path $o=x_0\sim \ldots x_n=o'$ from $o$ to $o'$, and let $f\in D^{p}$. By using Hölder's inequality, we get 
	
	\begin{align*}
	\abs{\nabla_{o,o'}f}&\leq \sum_{i=0}^{n-1}\abs{\nabla_{x_i,x_{i+1}}f}
	\leq \left(  \sum_{i=0}^{n-1} b^{-\frac{1}{p-1}}(x_i,x_{i+1}) \right)^{\frac{p-1}{p}}\cdot \left( \sum_{i=0}^{n-1}b(x_i,x_{i+1})\abs{\nabla_{x_i,x_{i+1}}f}^p\right)^{\frac{1}{p}}\\
	&\leq \left(  \sum_{i=0}^{n-1} b^{-\frac{1}{p-1}}(x_i,x_{i+1}) \right)^{\frac{p-1}{p}}\cdot \E_p^{1/p}(f).
	\end{align*}
	Define  $C_p(o,o')$ as the infimum of $ \left(\sum_{i=0}^{n-1} b^{-\frac{1}{p-1}}(x_i,x_{i+1})\right)^{p-1} $ over all paths from $o$ to $o'$. Then, we conclude
\begin{equation}
\label{nablaestimate}
\abs{\nabla_{o,o'}f}^p\leq C_p(o,o')\E_p(f).
\end{equation}
{Thus}, using the $p$-triangle inequality $\abs{\alpha + \beta}^p\leq 2^{ p-1}\left(\abs{\alpha}^p+\abs{\beta}^p \right)$ for all $\alpha, \beta \in \RR$, we have
	\begin{align*}
	\norm{f}_{o,p}^p
	&\leq \E_p(f)+2^{ p-1}(\abs{\nabla_{o,o'}f}^p+\abs{f(o')}^p)
	\leq (1+2^{ p-1}C_p(o,o'))\E_p(f)+2^{ p-1}\abs{f(o')}^p\\
	&\leq C'\norm{f}_{o',p}^p,
	\end{align*}
	 where $C'=(1+2^{p-1}C_p(o,o'))\vee 2^{p-1}$. By a symmetry argument, we get the equivalence.

The last statement can be seen as follows: Clearly, $\operatorname{ev}_x$ is linear. Since $\abs{\operatorname{ev}_x(f)}^p=\abs{f(x)}^p\leq \norm{f}_{x,p}^p\leq C_p \norm{f}_{o,p}^p$ for some positive constant $C_p$, the statement follows.
\end{proof}
Note that the proof goes through by adding a non-negative potential to $\E_p$. By making a slight modification, the second statement of Lemma~\ref{lem:uniform} is also true for $p=1$. Note that some of the main results will explicitly use reflexivity, and thus exclude the case $p=1$.

The next lemma is a non-linear counterpart of \cite[Lemma~6.5]{KLW21} and is a consequence of the uniform convexity of the Banach space $(D^{p},\norm{\cdot}_{o,p})$ for all $p\in (1,\infty)$.
\begin{lemma}\label{lem:limsup}
	Let $o\in X$. Let $f\in D^{p}$ and let $(f_n)$ be a sequence in $D^{p}$. Then,  $f_n\to f$ pointwise and $\limsup_{n\to \infty}\E_p(f_n)\leq \E_p(f)$ if and only if $f_n\to f$ with respect to $\norm{\cdot}_{o,p}$.
\end{lemma}
\begin{proof}
$\implies$: The assumptions $f_n\to f$ pointwise and $\limsup_{n\to \infty}\E_p(f_n)\leq \E_p(f)$ imply
\[ \limsup_{n\to \infty}\norm{f_n}_{o,p}^p= \limsup_{n\to \infty}(\E_p(f_n)+\abs{f_n(o)}^p)\leq \E_p(f)+\abs{f(o)}^p=\norm{f}^p_{o,p}, \]
i.e., $(f_n)$ is bounded in the reflexive Banach space $(D^{p},\norm{\cdot}_{o,p})$. By \cite[Theorem~3.18]{Brezis11}, every ball in a reflexive Banach space is weakly sequentially compact. Since $(f_n)$ converges pointwise to $f$, every subsequence of $(f_n)$ has a subsubsequence which converges weakly to $f$. Therefore, $(f_n)$ converges weakly to $f$. By \cite[Proposition~3.32]{Brezis11}, the uniform convexity together with weak convergence and $\limsup_{n\to \infty}\norm{f_n}_{o,p}\leq \norm{f}^p_{o,p}$ implies that
 $(f_n)$ converges strongly to $f$, i.e., $\norm{f-f_n}_{o,p}\to 0$.

$\Longleftarrow$: By Lemma~\ref{lem:uniform}, the point evaluation $\delta_x$ is continuous for all $x\in X$. Hence, since $f_n\to f$ with respect to the (continuous) norm $\norm{\cdot}_{o,p}$, we also have $f_n\to f$ pointwise. Hence, $\E_p(f_n)= {\norm{f_n}_{o,p}^p- \abs{f_n(o)}^p\to \norm{f}_{o,p}^p- \abs{f(o)}^p}=\E_p(f)$, and the statement follows.
\end{proof}

For $p\in (1,\infty)$, we define the \emph{(variational) $p$-capacity} with respect to $V\sse X$ and $K \sse V$ finite by the formula
\[
\cc_p(K,V):=\inf\set{\E_p(\phi) :	\phi\in C_c(V), \, \phi \geq 1 \text{ on } K}.
\]
If $K=\set{o}$ is a singleton, $o\in X$, then we shorten notation and set $\cc_p(o,V):=\cc_p(\set{o},V)$. In the case where $V=X$ we will simply write $\cc_p(K)$ instead of $\cc_p(K,X).$
Note that the infimum in the definition of  capacity can be taken in the closure of $C_c(V)$ with respect to the norm
$\norm{\cdot}_{o,p}$.  Moreover, since $\E_p$ is Markovian, one can also restrict to functions satisfying $0\leq \phi\leq 1$ which are $1$ on $K$.

In particular, if $V=X$ we then have
\[
\cc_p(K)=\inf\set{\E_p(v) : v\in D_{0}^{p}, \,  0\leq v\leq1, \,	v= 1 \text{ on } K}.
\]

In the following lemma we collect some well-known properties of the $p$-capacity, and they are actually also valid for $p=1$.

\begin{lemma} Let $V_1,V_2 \sse X$ be arbitrary and $K_1, K_2\sse X$ finite and non-empty. Then the following properties hold:
\begin{enumerate}[label=(\alph*)]
\item\label{cap2}   the $p$-capacity is monotone, i.e.,
\begin{enumerate}[label=(\alph{enumi}\,\alph*)]
	\item\label{cap21} 	if $K_1\sse V_1\sse V_2$ then $\cc_p(K_1,V_1)\geq \cc_p(K_1,V_2)$,
	\item\label{cap22} if $K_1\sse K_2\sse V_1$ then $\cc_p(K_1,V_1)\leq \cc_p(K_2,V_1)$,
\end{enumerate}
\item\label{cap3} the $p$-capacity is strongly subadditive, i.e., if $K_1\cup K_2\sse V_1$ then
\[\cc_p(K_1\cup K_2,V_1)+\cc_p(K_1\cap K_2,V_1)\leq  \cc_p(K_1,V_1)+\cc_p(K_2,V_1);\]
\item\label{cap6} the $p$-capacity is determined on the boundary, more precisely, for $K_1\cup \partial_e K_1\sse V_1$ finite, we have
\[ \cc_p(K_1\cup \partial_e K_1, V_1)= \cc_p(\partial_e K_1, V_1).\]
\end{enumerate}
\end{lemma}
\begin{proof} The proof of \ref{cap2} is clear by the monotonicity of the infimum.
	
For \ref{cap3}, let $\epsilon >0$. Then there are $0\leq \phi_i\in C_c(V_1)$ such that $\phi_i= 1$ on $K_i$ and $\E_p(\phi_i)\leq \cc_p(K_i,V_1)+\epsilon$, $i=1,2$. By Lemma~\ref{lem:BeurlingDeny}, we have 
\[\E_p(\phi_1\wedge \phi_2)+\E_p(\phi_1\vee \phi_2)\leq \E_p(\phi_1)+\E_p(\phi_2).\]	
Hence, 
\begin{multline*}
	\cc_p(K_1\cup K_2,V_1)+\cc_p(K_1\cap K_2,V_1)
	 \leq \E_p(\phi_1\vee \phi_2) +\E_p(\phi_1\wedge \phi_2)\\
	 \leq \E_p(\phi_1)+\E_p(\phi_2)
	 \leq \cc_p(K_1,V_1)+\cc_p(K_2, V_1)+2\epsilon.
\end{multline*}
Taking $\epsilon \to 0$ finishes the proof.

Finally, by \ref{cap22}, ``$\geq$'' holds in \ref{cap6}. To prove the reverse inequality, given $\epsilon > 0$, we can find $\phi\in C_c(V_1)$ such that $0\leq \phi\leq 1$, $\phi=1$ on $\partial_e K_1$ and $\E_p(\phi)< \cc_p(\partial_e K_1, V_1)+\epsilon$. Let $\psi\in C_c(V_1)$ be such that $\psi= \phi$ on $X\setminus K_1$ and $\psi=1$ on $K_1$. In particular, $\psi = 1$ on $K_1\cup\partial_e K_1$ and it is easily checked that  $\abs{\nabla_{x,y}\psi}^p\leq \abs{\nabla_{x,y}\phi}^p$ for all $x\sim y\in X$. Hence, we get
\[ \cc_p(K_1\cup \partial_e K_1, V_1)\leq \E_p(\psi)\leq \E_p(\phi)< \cc_p(\partial_e K_1, V_1)+ \epsilon.\]
Letting $\epsilon \to 0$ yields the result.
\end{proof}
We remark that the first two statements in the lemma  also hold for non-negative energy functionals whereas the third statement uses the non-negativity of the potential. Also, essentially the same argument used to prove \ref{cap6} shows that, for every finite set $K\sse V$, $\cc_p(K, V)=\cc_p(\partial_i K, V). $

\begin{definition}\label{def_p-parabolicity}
A graph $G$ is said to be \emph{$p$-parabolic} if $\cc_{p}(K)=0$ for every 
finite subset $K$ of $X$. Otherwise, the graph is called \emph{$p$-hyperbolic}.
\end{definition}
\begin{remark}
	The word ``parabolic'' in Definition~\ref{def_p-parabolicity} has its origin in the geometry of surfaces. By the celebrated uniformisation theorem of Klein, Koebe and Poincaré in the linear ($p=2$)-case, any simply connected Riemann surface is conformally equivalent to either the sphere (surface of elliptic type), the Euclidean plane (surface of parabolic type) or the hyperbolic plane (surface of hyperbolic type). Since the sphere is the only one with compact surface, the \emph{type problem} is to decide whether a surface is of hyperbolic or of parabolic type. See also  \cite{Gri99} and references therein for more details. 
\end{remark}

The following is an immediate consequence of the definition.
\begin{lemma}
\label{parabolicity = null seq}
{For every $p\in (1,\infty)$}, the following are equivalent:
\begin{enumerate}[label=(\roman*)]
\item\label{thm:parab1} the graph $G=(X,b,m)$ is $p$-parabolic;
\item\label{thm:parab2} for every  finite set  $ K $ and for every $\alpha>0$ there exists $ (e_n) \sse C_c(X) $ such that $ e_n=\alpha$ on $ K $ and $ \E_p(e_n) \to 0 $ as $ n \to \infty $;
\item\label{thm:parab2'} there exists $  x  \in X $, $\alpha>0$ and $ (e_n) \sse C_c(X) $ with $ e_n(x) = \alpha$ for every $ n $ and $ \E_p(e_n) \to 0 $ as $ n \to \infty $;
\end{enumerate}
A sequence  $ (e_n) $ in $ C_c(X) $ satisfying condition \ref{thm:parab2'} is called a \emph{null-sequence}.
\end{lemma}

{We note that  in the above statement arbitrary potentials could be added as long as the $p$-energy functional remains non-negative, and also $p=1$ could be included. }

\begin{proof}
The implication \ref{thm:parab1} $\implies $ \ref{thm:parab2} follows directly from the definition of  $\cc_{p}(K)=0$ and the $p$-homogeneity of $\E_p$.

\noindent
To prove \ref{thm:parab2} $ \implies$ \ref{thm:parab1}  it suffices to consider $ \frac{1}{\alpha} e_n $ which is $ 1 $ on $K$ and {satisfies}
\begin{equation*}
\E_p \left( \alpha^{-1} e_n \right) = \alpha^{-p} \E_p(e_n) \to 0, \qquad p\in [1,\infty),
\end{equation*}
as $ n\to \infty $.

The implication \ref{thm:parab2} $\implies$ \ref{thm:parab2'} is trivial.

Finally, to show that \ref{thm:parab2'} $\implies$ \ref{thm:parab2}, let $(e_n) $ be a sequence as in \ref{thm:parab2'}. We claim that $ e_n \to \alpha $ uniformly on every finite subset of $ X $. Indeed, given $y \in X $, let $x= x_0\sim x_1 \sim \ldots \sim x_k=y $ be a path joining $ x$ to $ y$. Then, as in the proof of Lemma~\ref{lem:uniform}, for all $p\in [1,\infty)$ there is a constant $0<C(x,y)<\infty$ such that
\begin{align*}
\abs{\nabla_{x,y}e_n}^p&\leq C(x,y)\E_p(e_n). 
\end{align*}
Since $\E_p(e_n)\to 0 $ as $ n \to \infty$, we have $\min_{K} e_n \to \alpha$. Hence, there exists $n_0\in \NN$ such that for all $n\geq n_0$ we have $\min_K e_n > 0$.

Now let us consider $ \phi_n=\beta (\frac{e_n}{\min_K e_n} \wedge 1) $ for all $n\geq n_0$ and some $\beta >0$. Since $\E_p$ is Markovian and cutting at one is a normal contraction, we get
\[ \E_p(\phi_n)= \beta^p \E_p (\frac{e_n}{\min_K e_n}\wedge 1)\leq \beta^p \E_p(\frac{e_n}{\min_K e_n})= \frac{\beta^p}{(\min_K e_n)^p} \E_p(e_n)\to 0 \]
as $n\to \infty$. Moreover, by construction, $\phi_n = \beta$ on $ K$. Thus, $(\phi_n)_{n\geq n_0}$ is the desired null-sequence.
\end{proof}
The following is just an immediate consequence of Lemma~\ref{parabolicity = null seq}.
\begin{corollary}\label{cor:cap}
	The graph is $p$-parabolic if and only if there exists $x\in X$ such that $\cc_p(x)=0$.
\end{corollary}

\section{Characterizations of $p$-Parabolicity}\label{s:chara}
In this section we provide several characterizations  of $p$-parabolicity for summable infinite graphs. In the case of locally finite graphs, some of these results were proven in \cite{Y77,Y84,Soa2, Soa1, SY93Class, SY93Para, Prado}. However, replacing the assumption of local finiteness with mere summability introduces non-trivial difficulties, and therefore this section represents a novelty both in the context and the content. For non-locally finite graph in the linear setting $p=2$ see e.g. \cite[Chapter~6]{KLW21} and references therein.

\subsection{Green's Functions}\label{s:Green}
Recall that $p\in (1,\infty)$.
In the continuous setting, it is classical that $p$-hyperbolicity is also  equivalent to the existence of a Green's function with pole at $o\in X$, for some/all $o\in X$, that is, a minimal positive solution to the problem
\begin{equation}
\label{Green-def}
\Delta_p  g(x,o)= \hat{\mathds{1}}_o(x):= \frac 1{m(x)} \mathds{1}_o(x).
\end{equation}
A Green's function, when it exists, is closely related to non-constant  solutions of the following problem
\begin{equation}\label{p-harmonic potential}
\begin{cases}
\Delta_p u= 0& \text{in } X\setminus \{o\};\\
u(o)=1,  \quad  &0\leq u\leq 1.
\end{cases}
\end{equation}
Indeed, if  $u$ is a non-constant solution to \eqref{p-harmonic potential}, then $\Delta_p u(o)>0$ and it follows that a suitable multiple of $u$ is a positive solution to \eqref{Green-def}.   Minimality is then achieved by performing a suitable translation. Note also that in this case, $u$ is the $p$-harmonic potential for $(o,X)$, i.e., a  minimizer for $\cc_p(o)$.

In the following proposition we collect some of the properties of solutions to  \eqref{p-harmonic potential}.
The proof uses uniform convexity.

\begin{proposition}
\label{existence potential}
Let $o\in X$ and let $\{X_n\}$ be an increasing exhaustion of $X$ by finite and connected sets such that $o\in X_1$ and
$X_{n+1}\setminus X_n \subset \partial_e X_n$.
\begin{enumerate}[label=(\alph*)]
\item\label{existence potential i} For every $n$ there exists $u_n\in C(X)$ such that
\[
\begin{cases}
\Delta_p u_n=0\,\text{ in }  X_n\setminus \{o\}&\\
u_n(o)=1, \,\, u_n=0 \text{ on } X\setminus X_n,&
\end{cases}
\]
and
\[
\E_p(u_n)=\cc_p(o, X_n).
\]
Moreover $0\leq u_n\leq u_{n+1}\leq 1$.
\item \label{existence potential ii}
As $n\to \infty$, $u_n$ tends monotonically from below pointwise and in $D_{0}^{p}$ to a function $u\in D_{0}^{p}$, referred to as the $p$-harmonic potential at $o$, which satisfies \eqref{p-harmonic potential},  $0< u\leq 1$ and for which
\[
\E_p(u)=\cc_p(o).
\]
Moreover, $u$ is the smallest positive supersolution to \eqref{p-harmonic potential}, i.e., for all positive functions $v\in\FF^p$ such that $\Delta_{p}v\geq 0$ on $X\setminus \set{o}$ and $v(o)\geq 1$, we have $v\geq u$.
\item\label{existence potential iii} $u$ is constant if and only if $\cc_p(o)=0$ if and only if $\Delta_p u(o)=0$. Moreover, if $u$ is not constant then  $\inf u=0$.
\item\label{existence potential iv} The Green's function with pole at $o$ exists if and only if $u$ is not constant and in that case
$g(x,o)=cu(x)$ with $c=[m(o)\Delta_pu(o)]^{-1/(p-1)}$. In particular, $g(x,o)\in D_{0}^{p}$ and $g(o,o)^{-1/(p-1)}g(x,o)$ is the $p$-harmonic potential of $(o,X)$.  
\end{enumerate}
\end{proposition}
\begin{proof}
The proof follows from standard variational arguments. For completeness and the convenience of the reader, we provide an outline of the arguments.

\noindent
\ref{existence potential i}: Since the set
\[
\mathcal{K}_n=\{\phi \in C(X)\,:\, \phi(o)=1,\,\, \phi=0 \,\text{ on }\, X\setminus X_n\}
\]
is closed and convex in the uniformly convex space $(D_{0}^{p}, \norm{\cdot}_{o,p})$ and  since $\E_p$ and $\norm{\cdot}_{o,p}$ have the same minimizers on $\mathcal{K}_n$, $\E_p$ admits a (unique) minimizer $u_n$ which satisfies $u_n(o)=1$ and $\cc_p(o,X_n)= \E_p(u_n)\leq \E_p(\phi)$ for all $\phi \in \mathcal{K}_n$. For every $\phi_0$ supported in $X_n\setminus \{o\}$, we have $u_n+t\phi_0\in \mathcal{K}_n$ and
\[
\frac{\E_p(u_n+t\phi_0)-\E_p(u_n)}t = \frac{1}{2}\sum_{x,y} b(x,y)\frac{|\nabla_{x,y}u_n+t\nabla_{x,y}\phi_0|^p-|\nabla_{x,y}u_n |^p}t.
\]
As $t\to 0$, the fraction inside the summation tends to
 \[
 p|\nabla_{x,y} u_n|^{p-2}\nabla_{x,y} u_n \nabla_{x,y} \phi_0 ,
 \]
 and an application of the mean value theorem together with Young's inequality shows that
 \[
 \left|\frac{|\nabla_{x,y} u_n + t \nabla_{x,y} \varphi_0|^p-|\nabla_{x,y} u_n|^p}t\right|
 \leq
 C_p( |\nabla_{x,y} u_n |^p +|\nabla_{x,y} \varphi_0|^p),
 \]
 so that, letting $t\to 0$ and applying dominated convergence yield
 \begin{equation}
 \label{variation}
0=\frac{1}{2}\sum_{x,y} b(x,y) \p{\nabla_{x,y}u_n}\nabla_{x,y}\phi_0=\E_{p,X}(u_n,\phi_0).
 \end{equation}
In particular, letting $\phi_0=\mathds{1}_{x_o}$ with $x_o\in X_n\setminus \{o\}$ shows that
\[
\Delta_p u_n(x_o)=0.
\]
The last statement in \ref{existence potential i} follows by comparison, see Lemma~\ref{lem:WCPNonNegativePotential}.

\noindent
\ref{existence potential ii}: Replacing $\mathcal{K}_n$ with
\[\mathcal{K}=
\{v\in D_{0}^{p}\,:\, v(o)=1,\,0\leq v \leq 1 \}
\]
and arguing as above shows that there exists a unique $u\in \mathcal{K}$ such that $\E_p(u)=\cc_p(o)$, and $u$ satisfies \eqref{p-harmonic potential} and by the strong minimum principle $0< u\leq 1$, see e.g. Lemma~\ref{max_principle}. By comparison, $0\leq u_n\leq u$ so that the sequence
$(u_n)$ constructed in \ref{existence potential i} converges monotonically to a function $\tilde u$ satisfying $u_n\leq \tilde u\leq u$. By Fatou's lemma
\[
\E_p(\tilde u)\leq \lim_n \E_p(u_n)=\lim_n \cc_p(o,X_n)=\lim_n\cc_p(o)=\E_p(u),
\]
so, by uniqueness of the minimizer, $u=\tilde u$. It follows from Lemma~\ref{lem:limsup} that 
$u_n\to u$ in  $\norm{\cdot}_{o,p}$-norm,  as required.

That $u$ is the smallest positive supersolution to \eqref{p-harmonic potential} is another consequence of the comparison principle, Lemma~\ref{lem:WCPNonNegativePotential}. Let $v$ be another supersolution. Since $X_n$ is finite, $(v-u_n)\wedge 0$ attains its minimum on $X_n$. By Lemma~\ref{lem:WCPNonNegativePotential}, $u_n\leq v$ on $X$. Since $u_n \nearrow u$, we get $u\leq v$.

\noindent
\ref{existence potential iii}: Since $\E_p(u)=\cc_p(o)$ and, by connectedness, $u$ is constant if and only if $\E_p(u)=0$, the first assertion follows immediately.  It is obvious that if $u $ is constant then $\Delta_p u(o)=0$. On the other hand, if $\Delta_p u(o)=0$ then $u$ is $p$-harmonic on $X$ and, since $u\leq 1$ and $u(o)=1$, $u$ is constant by the strong maximum principle. To prove the last assertion, assume that $u$ is not constant and let  $\inf u=\gamma$. Then  $\hat u = (1-\gamma)^{-1} (u-\gamma)$ is a positive solution of \eqref{p-harmonic potential} with $\inf \hat u=0$. The comparison argument above shows that  
$ u\leq \hat u$ so that $\inf  u=\inf \hat u=0$.

\noindent
\ref{existence potential iv}: Assume first that the function $u$ constructed in \ref{existence potential ii} is not constant. Since $u(o)=1$ while $u(x)\leq 1$, $\Delta_p u(o)\geq 0$ and therefore $\Delta_pu(o)>0$ by \ref{existence potential iii}. A straightforward computation shows that, if $c=[m(o)\Delta_pu(o)]^{-1/(p-1)}$, then the function $\tilde u=c u(x)$   is non-negative, satisfies  $\Delta_p  \tilde u= \hat{\mathds{1}}_o$ and $\inf \tilde u =0$ so that $g(x,o)=\tilde u(x)$ is the required  minimal positive solution of \eqref{Green-def}.  For the converse, assume that there exists a minimal positive solution $g(x,o)$ to \eqref{Green-def}. We claim that if $c=g(o,o)$ then  $v=c^{-1}g(x,o)$ is the function $u$ constructed in \ref{existence potential ii} and therefore $u$ is not constant and $g(x,o)\in D_{0}^{p}$. To prove the claim, note that it follows from the usual comparison argument that $u_n\leq  v$ for every $n$, so that $ u=\lim_n u_n\leq v$. Since $u(o)=v(o)=1$, $\Delta_p u(o) \geq \Delta_p v(o)>0$. It follows that $\hat u = \left ( \Delta_p v(0)/\Delta_p u(0)\right)^{1/(p-1)} u$ satisfies
\[
\Delta_p \hat u= \Delta_p v\quad \text{and}\quad \hat u\leq u\leq v.
\]
We conclude that $0\leq c\hat u\leq c v=g(x,o)$ satisfies $\Delta_p(c \hat u)=\Delta_p g(\cdot, o)$. By minimality, $c\hat u=g(\cdot,o)$ and the claim follows.
\end{proof}
By using a different strategy a Green's function has been constructed in \cite{F:AAP} for more general $p$-Schrödinger operators.

Combining Lemma \ref{parabolicity = null seq} and Proposition \ref{existence potential}  we obtain the following classical characterizations of $p$-parabolicity in terms of existence of null sequences,  non-existence of Green's functions and vanishing of nonnegative superharmonic functions.

{\begin{theorem}\label{thm:criticalNew}
	The following statements are equivalent:
	\begin{enumerate}[label=(\roman*)]
		\item\label{thm:critical1New} The graph $G=(X,b,m)$ is $p$-parabolic.
\item\label{thm:critical2New} For any $o\in X$ and $\alpha>0$ there is a null-sequence $(e_n)$ in $X$ such that $e_n(o)=\alpha$, $n\in\NN$.	
		\item\label{thm:critical9New} For some/every $o\in X$  there does not exist a Green's function with pole at $o$.
\item\label{thm:Liouville p-superharmonic} If $u\geq 0$ satisfies $\Delta_p u\geq 0$ on $X$ then $u$ is constant.
	\end{enumerate}
\end{theorem}}

\begin{proof}
\ref{thm:critical1New} $\iff$ \ref{thm:critical2New} is the content of Lemma \ref{parabolicity = null seq}, while the equivalence
\ref{thm:critical1New} $\iff$ \ref{thm:critical9New} follows from Proposition~\ref{existence potential}.
Finally, since a Green's function is a nonnegative superharmonic function, clearly  \ref{thm:Liouville p-superharmonic} implies \ref{thm:critical9New} again by Proposition~\ref{existence potential}. On the other hand, let $v$ be  $p$-superharmonic and non-negative. By the strong minumum principle, Lemma~\ref{max_principle}, $v$ is stricly positive and by truncating and multiplying it by a constant we may assume that $0<v\leq 1$ on $X$. Arguing as in the proof of Proposition \ref{existence potential} shows that $1\ge v\ge u$ is the function constructed there. If $X$ is $p$-parabolic, $u\equiv 1$ and therefore also $v\equiv 1$, as required to conclude.
\end{proof}

\subsection{The Functional-Theoretic Perspective}\label{s:func}

The following is the main result of this subsection and the generalisation of \cite[Section~6.2]{KLW21} from $p=2$ to $p\in (1,\infty)$. We remark that equivalence (\ref{D2} $\iff$ \ref{D3}) was already observed  in \cite{Y77}, and (\ref{D2} $\iff$ \ref{D9}) in \cite{MY92}, both in the setting of locally finite graphs.

\begin{proposition}[$\DD^p=\DD_{0}^{p}$]\label{prop:D=D_{0}^{p}}
	The following are equivalent:
	\begin{enumerate}[label=(\roman*)]
	\item\label{D1} The graph $G=(X,b,m)$ is $p$-parabolic.
	\item\label{D2} $\DD^p=\DD_{0}^{p}$.
\item\label{D2A} $\DD^p\cap \ell^p(X,\tilde{m})=\DD_{0}^{p}\cap \ell^p(X,\tilde{m})$ for all measures $0<\tilde{m}\in C(X)$.
	\item\label{D3} $1\in \DD_{0}^{p}$.
	\item \label{D4} There exists a null-sequence $(e_n)$ in $C_c(X)$ such that $0 \leq e_n \nearrow 1$ pointwise.
	\item\label{D5} There exists a sequence $(e_n)$ in $C_c(X)$ such that $e_n\to 1$ pointwise and $\sup_n \E_p(e_n)< \infty$.
	\item\label{D6} There exists $u\in \DD_{0}^{p}$ and a finite set $K\sse X$ such that $\inf_{x\in X\setminus K}u(x)>0$.
	\item\label{D8} The functionals $\E_p^{1/p}$ and $\norm{\cdot}_{o,p}$ are not equivalent on $C_c(X)$ for some (or, equivalently, all) $o\in X$.
	\item\label{D9} $(D_0^p, \E_p^{1/p})$ is not a Banach space.
	\item\label{D7} The point evaluation $\operatorname{ev}_x\colon (\DD_{0}^{p}, \E_p^{1/p})\to \RR, \operatorname{ev}_x(f)=f(x)$ is not continuous for some (or, equivalently, all) $x\in X$.
	\end{enumerate}
\end{proposition}

\begin{proof}
\noindent \ref{D1} $\iff$ \ref{D4}:  By Theorem \ref{thm:criticalNew}, \ref{D4} implies \ref{D1}. For the reverse implication, if \ref{D1} holds then, again by Theorem \ref{thm:criticalNew}, there exists a null-sequence $(e_n)$ in $C_c(X)$ such that $e_n(o)=1$ for all $n$. By \eqref{nablaestimate} in Lemma \ref{lem:uniform}, for all $x\in X$
\[
|1-e_n(x)|^p=|e_n(o)-e_n(x)|^p=|\nabla_{o,x} e_n|^p\leq C_p(o,x)\E_p(e_n)\to 0,
\]
and therefore $e_n\to 1$ pointwise on $X$. Since $\E_p$ is Markovian and $\cdot \wedge 1$ is a normal contraction, $(e_n\wedge 1)$ is a null-sequence which converges from below to $1$. By the monotone subsequence theorem and a diagonal argument, we get an increasing null-sequence.

\noindent\ref{D2} $\implies$ \ref{D2A}: This is trivial.

\noindent\ref{D2A} $\implies$ \ref{D3}: Since \ref{D2A} holds for all measures, it does also hold for a finite measure $m$. Thus, $1\in \ell^p(X,m)$. Since clearly $1\in \DD^p$, we conclude
\[1\in \DD^p\cap \ell^p(X,m)= \DD_{0}^{p}\cap \ell^p(X,m)\sse \DD_{0}^{p}.\]

\noindent\ref{D3} $\implies$ \ref{D4}: Since $1\in \DD_{0}^{p}$ there exists a sequence $(\phi_n)$ in $C_c(X)$ such that $\norm{\phi_n-1}_{o,p}\to 1$. This implies that $\E_p(\phi_n -1)\to 0$, i.e., $\phi_n \to 1$ pointwise, and $\E_p(\phi_n)=\E_p(\phi_n-1)\to 0$. Moreover, consider the normal contraction $N(x)=0 \vee x \wedge 1$, $x\in \RR$. Since $\E_p$ is Markovian, $(N \circ \phi_n)$ is the desired null-sequence.

\noindent\ref{D4} $\implies$ \ref{D5}: Trivial.

\noindent\ref{D5} $\implies$ \ref{D3}: This is an application of the Banach-Saks theorem. Here are the details: By assumption, $(e_n)$ is bounded in the uniformly convex Banach space $(\DD^p,\norm{\cdot}_{o,p})$. By weak sequential compactness and by a celebrated result of Kakutani \cite{K38}, stating that uniformly convex Banach spaces have the Banach-Saks property,  $(e_n)$ has a subsequence, still denoted by $(e_n)$, whose Cesaro means $\phi_n=n^{-1}\sum_{1}^ne_k$ converge to $u$ strongly in $(\DD^p,\norm{\cdot}_{o,p})$. Since $e_n\to 1$ pointwise, also $\phi_n\to 1$ pointwise. Moreover, convergence in $\norm{\cdot}_{o,p}$ implies pointwise convergence by Lemma~\ref{lem:uniform} and we get $1=u\in \DD_{0}^{p}$.

\noindent\ref{D3} \& \ref{D4} $\implies$ \ref{D2}: Let $(e_n)$ be a null-sequence in $C_c(X)$ such that $0\leq e_n\nearrow 1\in \DD_{0}^{p}$ pointwise.

First we show  that every bounded function $f\in \DD^{p}$ can be approximated in $\DD_{0}^{p}$. Since $e_nf\in C_c(X)$ we get using the $p$-triangle inequality,
\begin{multline*}
	\E_p(f-e_nf)\\
	\leq 2^{p-1}\left(\sum_{x\in X}(1-e_n(x))^p\sum_{y\in X}b(x,y)\abs{\nabla_{x,y}f}^p+\sum_{y\in X})\abs{f(y)}^p\sum_{x\in X}b(x,y)\abs{\nabla_{x,y}e_n}^p\right) \\
	\leq 2^{p-1}\left(\sum_{x\in X}(1-e_n(x))^p\sum_{y\in X}b(x,y)\abs{\nabla_{x,y}f}^p+2\norm{f}_\infty^p \E_p(e_n)\right).
\end{multline*}
Now the second term vanishes since $(e_n)$ is a null-sequence, and the first term vanishes using dominated convergence. Hence, $\norm{f-e_nf}_{o,p}\to 0$ as $n\to \infty$, i.e., $\DD^p\cap \ell^\infty(X) \sse \DD_{0}^{p}$.

Next, take $f\in \DD^p$ and define $f_n=-n\vee f\wedge n$, $n\in \NN$. Then $f_n\in \DD^p\cap \ell^\infty(X)$. Clearly, $f_n\to f$ pointwise, and since $\E_p$ is Markovian, i.e., compatible with normal contractions, we have $\limsup_{n\to\infty}\E_p(f_n)\leq \E_p(f)$. Thus, we can use Lemma~\ref{lem:limsup} and get that $f_n\to f$ with respect to $\norm{\cdot}_{o,p}$, i.e., $\DD^p\sse \DD_{0}^{p}$.

\noindent\ref{D3}$\iff$ \ref{D6}: If $1\in \DD_{0}^{p}$, then \ref{D6} is clearly satisfied for $u=1$. On the other side, assume that \ref{D6} holds. Since $\cdot \vee 0 $ is a normal contraction, we may assume that $u\ge 0$. Next, $v:=u+1_K\in \DD_{0}^{p}$. Set $m:=\inf_{X\setminus K} u \wedge 1 >0$. Then, $(v/m) \wedge 1=1$. Moreover, $\cdot \wedge 1$ is a normal contraction, and thus $\E_p(1)=\E_p((v/m)\wedge 1)\leq \E_p(v/m)< \infty$, i.e., $1\in \DD_{0}^{p}$.

\noindent \ref{D3}$\implies$ \ref{D8}: If $\E_p^{1/p}$ and $\norm{\cdot}_{o,p}$ are equivalent on $C_c(X)$, then they are also equivalent on $\DD_{0}^{p}$. So, if $1\in \DD_{0}^{p}$, then $\E_p(1)=0$ but $\norm{1}_{o,p}=1$ and they are not equivalent.

\noindent \ref{D8}$\implies$ \ref{D1}: According to Corollary~\ref{cor:cap}, if the graph is $p$-hyperbolic, then $\cc_p(o)>0$ for all (or some) $o\in X$. Let  $\phi\in C_c(X)$. Since, by definition of capacity, $\cc_p(o)|\phi(o)|^p\leq \E_p(\phi)$, we have
\[\cc_p(o)\norm{\phi}_{o,p}^p= \cc_p(o)(\E_p(\phi)+\abs{\phi(o)}^p)\leq (\cc_p(o)+1)\E_p(\phi). \]
Since, clearly, $\E_p\leq \norm{\cdot}_{o,p}^p$ on $\DD_{0}^{p}$, both functionals are equivalent.

\noindent \ref{D9}$\implies$ \ref{D8}: We show the contraposition. 
If both $\E_p^{1/p}$ and $\norm{\cdot}_{o,p}$ are equivalent on $C_c(X)$ for all $o\in X$, they are also equivalent on $D_0^p$. Therefore, $\E_p^{1/p}$ is a norm on $D_0^p$. Let now $(f_n)$ be a Cauchy-sequence in $(D_0^p,\E_p^{1/p})$, i.e., $f_n\in D_0^p$ and $\E_p(f_n-f_m)\to 0$ as $n,m\to\infty$. Hence, also $\norm{f_n-f_m}_{o,p}\to 0$, and therefore $(f_n(o))_n$ is bounded. By Lemma~\ref{lem:uniform} and the connectedness of the graph, $(f_n(x))_n$ is bounded for all $x\in X$. Thus, $f_n\to f\in C(X)$ pointwise. Since $\E_p$ is lower semi-continuous and $(f_n)$ is Cauchy, we have $\E_p(f)\leq \liminf_{n\to\infty}\E_p(f_n)< \infty$. Thus, $f\in D_0^p$. Since $\E_p(f-f_n)\leq \liminf_{m\to\infty}\E_p(f_m-f_n)\to 0$ as $n\to\infty$, we conclude that $(D_0^p, \E_p^{1/p})$ is complete.

\noindent \ref{D3}$\implies$ \ref{D9}: This is clear since $\E_p^{1/p}(1)=0$, i.e., $(D_0^p,\E_p^{1/p})$ is only a semi-normed space.

\noindent \ref{D7}$\implies$ \ref{D8}: We show the contraposition and assume that $E^{1/p}$ and $\norm{\cdot}_{o,p}$ are equivalent on $C_c(X)$ and thus, also on $D_0^p$. Then, by Lemma~\ref{lem:uniform} $\operatorname{ev}_x$ is bounded on $(D_0^p, \E_p^{1/p})$.

\noindent \ref{D1}$\implies$ \ref{D7}: We show that $\operatorname{ev}_x$ is not bounded for any $x\in X$. From Lemma~\ref{parabolicity = null seq}, we get the existence of a null-sequence $(e_n)$ such that $e_n(x)=1$, and $\E_p(e_n)\to 0$. Hence,
\[\norm{\operatorname{ev}_x}=\sup_{0\neq \phi\in D_0^p}\frac{\abs{\operatorname{ev}_x(\phi)}}{\E_p^{1/p}(\phi)}\geq  \frac{1}{\E_p^{1/p}(e_n)}\to \infty.\qedhere\]
\end{proof}

A direct consequence is the following result. It connects the capacity with the Hardy inequality and a closely related inequality. For the linear case see \cite{S23} or also \cite{Fuk}. We note that this statement actually holds for general non-negative $p$-energy functionals \cite{F:Thesis}, and also $p=1$ would be possible. See also the recent results in \cite{SZ25}, and for locally finite graphs, see \cite{MY92}.

\begin{proposition}[Hardy Inequality]\label{prop:Hardy}
	The following are equivalent:
	\begin{enumerate}[label=(\roman*)]
	\item\label{H1} The graph is $p$-hyperbolic.
	\item\label{H0} There exists a (strictly) positive $\mu\in C(X)$, referred to as a Hardy weight,   such that for all $\phi\in C_c(X)$,
	\[\norm{\phi}_{p,\mu}\leq \E_p(\phi).\]
	\item\label{H2} There exists a (strictly) positive $\mu\in C(X)$ such that for all $\phi\in C_c(X)$,
	\[\norm{\phi}_{1,\mu}\leq \E_p^{1/p}(\phi).\]
	\end{enumerate}
\end{proposition}
\begin{proof} \ref{H1} $\implies$ \ref{H0}: For completeness, we reproduce the short argument  from \cite{F:GSR}. By definition, we have that $\cc_{p}(x)> 0$ for all $x\in X$. Moreover, $\cc_{p}(x)\cdot \mathds{1}_x$ is a possible $\mu$ since $\cc_{p}(x)\abs{\phi(x)}^{p}\leq \E_p(\phi)$ for all $\phi\in C_c(X)$. Furthermore, let $\alpha_x>0$ such that $\sum_{x\in X}\alpha_x=1$, then also $\sum_{x\in X}\alpha_x \cc_{p}(x)\cdot \mathds{1}_x$ is a possible $\mu$, and thus $\mu$ can {be} chosen to be strictly positive on $X$.

\noindent \ref{H0} $\implies$ \ref{H2}: By assumption, there exists a strictly positive Hardy weight. Moreover, every pointwise smaller but positive function serves as a Hardy weight as well. Hence, we can assume that our Hardy weight $\mu$ fulfils $\mu\in \ell^1(X,1)$, and even $\norm{\mu}_{1,1}\leq 1$. Then, since $p\in (1,\infty)$ by Hölder's inequality with exponents $p,q$ and  \ref{H0} we have
\[\norm{\phi}_{1,\mu}\leq \norm{\phi}_{p,\mu}\norm{\mu^q}_{1,1}^{1/q} \leq \E_p^{1/p}(\phi), \qquad \phi\in C_c(X).\]

\noindent \ref{H2} $\implies$ \ref{H1}: If the graph is $p$-parabolic, then there is a null-sequence $(e_n)$ of $\E_p$ going pointwise and monotonically to $1$ (Proposition~\ref{prop:D=D_{0}^{p}}). Thus, $\E_p^{1/p}(e_n)\to 0$ but $\norm{e_n}_{1,\mu}\to \norm{1}_{1,\mu}>0$ since $\mu\gneq 0$.
\end{proof}

Note that in the statement above the function space $C_c(X)$ can be enlarged to $\DD_0^p$ but not to $D^p$ (take e.g. $1\in D^p$).

\subsection{The Superharmonic Functions of Finite Energy Perspective}\label{s:super}
We show here the quasi-linear version of \cite[Theorem~6.25]{KLW21}. The  proof uses some approximation techniques which we will state next. With minor changes, these results can be proven along the lines of \cite[Lemma~6.6]{KLW21} and so we will skip the proof. Also adding a non-negative potential and including $p=1$ would be possible.

\begin{lemma}[Bounded and monotone approximation]\label{lem:approx}
	Let $u\in \DD^p$ and $(u_n)$ be a sequence in $\DD^p$ such that $\norm{u-u_n}_{o,p}\to 0$ as $n\to \infty$.
	\begin{enumerate}[label=(\alph*)]
	\item\label{lem:approxBounded} Define $v_n:= -u_- \vee u_n \wedge u_+$. Then, $\norm{u-v_n}_{o,p}\to 0$ as $n\to \infty$.
	\item\label{lem:approxMonotone} If, furthermore, $u\geq 0$ and $u_n\in C_c(X)$ for all $n\in \NN$, then  there exists a monotone increasing sequence $(\phi_n)$ in $C_c(X)$ such that $0\leq \phi_n \leq u$ and $\norm{u-\phi_n}_{o,p}\to 0$ as $n\to \infty$. On the set where $u>0$, we can choose $\phi_n$ such that $0\leq \phi_n< u$.
	\end{enumerate}
\end{lemma}

Before stating the main result of this subsection, we need another version of Green's formula. Recall that we defined 
\[ \E_{p}(u,v):=\frac{1}{2} \sum_{x,y\in X}b(x,y)\p{\nabla_{x,y}u}\nabla_{x,y}v\]
for all functions $u, v\in C(X)$ whenever the sum converges absolutely.

{We remark that the conclusions of the next lemma hold true adding a non-negative potential and for $p=1$.} 
\begin{lemma}[Green's formula on $\DD^p$]\label{lem:GreensFormulaD}
	Let $u\in \DD^p$ and either
	\begin{enumerate}[label=(\alph*)]
			\item\label{lem:GreenD>0} $v\in \DD_{0}^{p}$ and $\Delta_pu\geq 0$ on $X$, or
			\item\label{lem:GreenDl1} $v\in \DD_0^p\cap \ell^\infty(X)$ and $\Delta_pu\in \ell^1(X,m)$.
		\end{enumerate}
 Then the  sums defining the quantities in the formula converge absolutely and
	\[\E_{p}(u,v)=\ip{\Delta_pu}{v}_{X}.\]
\end{lemma}
\begin{proof}
{The proof follows the lines of the $(p=2)$-case in \cite[Lemma~6.8]{KLW21}. Since the arguments are quite short, we reproduce them here. }

If $v\in C_c(X)$ then we can simply apply Lemma~\ref{lem:GreensFormula} and get the result.

\noindent \ref{lem:GreenD>0}: Since $v\in \DD_{0}^{p}$ which is invariant under normal contractions,  $v_+, v_-\in \DD_{0}^{p}$. By Lemma~\ref{lem:approx}~\ref{lem:approxMonotone}, $v_{\pm}$ can be approximated monotonically by functions in $C_c(X)$. Hence, by monotone convergence we get \[\E_p(u,v_{\pm})=\ip{\Delta_pu}{v_{\pm}}_X.\]
Using the linearity of the difference operator, we obtain the desired formula.

 \noindent\ref{lem:GreenDl1}: By Lemma \ref{lem:approx}~\ref{lem:approxBounded}, we can approximate $v$ with functions with smaller absolute value. Using $\Delta_pu\in \ell^1(X,m)$ and dominated convergence yield the result.
\end{proof}

Now we can show the main result of this subsection, which can be seen as a generalised Liouville property. 
\begin{proposition}[Bounded superharmonic functions are constant]\label{prop:S}
		The following are equivalent:
		\begin{enumerate}[label=(\roman*)]
			\item\label{S1} The graph $G=(X,b,m)$ is $p$-parabolic.
			\item\label{S2} All non-negative superharmonic functions are constant.
			\item\label{S3} All superharmonic functions in $\DD^p$ are constant.
			\item\label{S4} All superharmonic functions in $\DD_{0}^{p}$ are constant.
			\item\label{S5} All bounded superharmonic functions are constant.
			\item\label{S7sup} All superharmonic functions which are bounded from below are constant.
			\item\label{S6} There exists a non-zero harmonic function in $\DD_{0}^{p}$.
		\end{enumerate}
\end{proposition}
\begin{proof}
\noindent\ref{S1} $\Leftrightarrow$ \ref{S2}: This is Theorem \ref{thm:criticalNew}.

\noindent\ref{S1} $\implies$ \ref{S3}: Let $u\in \DD^p$ be superharmonic. By Proposition~\ref{prop:D=D_{0}^{p}}, we have that $1\in \DD_{0}^{p}=\DD^p$. Thus using the Green's formula, Lemma~\ref{lem:GreensFormulaD}, we conclude that $0=\E_p(u,1)=\ip{\Delta_pu}{1}_X$, i.e., $u$ is harmonic.

Fix $o\in X$, we define the auxiliary function $v:=u-u(o)\cdot 1$. The goal is now to show that $v=0$. It is obvious that $v\in \DD^p$ is harmonic. Since $-\abs{\cdot}$ is concave, $-\abs{v}\in \DD^p$ is superharmonic (see \cite[Proposition~4.2]{F:Opti}) and hence, also harmonic. Thus, $\abs{v}$ is harmonic and nonnegative. Since $v(o)=0$, we conclude that $|v|=0$ by the strong minimum principle (Lemma~\ref{max_principle}), as required to show that $u$ is constant.

\noindent\ref{S3} $\implies$ \ref{S4}: This is trivial.

\noindent \ref{S4} $\implies$ \ref{S1}: By Proposition~\ref{existence potential}~\ref{existence potential ii}, there exists a positive superharmonic function $u\in \DD_{0}^{p}$ with $u(o)=1$ and $\E_p(u)=\cc_p(o)$. By \ref{S4}, $u=1\in D_{0}^{p}$, and by Proposition~\ref{prop:D=D_{0}^{p}}, we get \ref{S1}.

\noindent \ref{S2} $\implies$ \ref{S5}: Let $u\in \ell^\infty(X)$ be superharmonic. Then,  $u- \inf_{X}u \cdot 1$ is also a non-negative superharmonic function, and thus constant. Hence, $u$ is constant.

\noindent\ref{S5} $\implies$ \ref{S7sup}: {Note that if $u_i\in \FF^p(X)$, $i=1,2$, then $u_1\wedge u_2\in \FF^p(X)$ and it follows by a case  {by case} analysis that $\Delta_p (u_1\wedge u_2)\geq \Delta_p u_1\wedge \Delta_p u_2$. Thus, if $u$ is superharmonic and bounded from below, then, for every $k\in \NN$, $ u\wedge k$ is superharmonic and bounded. Hence it is constant.  It follows that  $u= \lim_k u\wedge k$ is constant. }

\noindent\ref{S7sup} $\implies$ \ref{S2}: This is trivial.

\noindent\ref{S6} $\implies$ \ref{S1}: Let $u\in \DD_{0}^{p}$ be harmonic and non-zero. By the Green's formula on $\DD^p$, Lemma~\ref{lem:GreensFormulaD} (a), we have $\E_p(u)=\ip{\Delta_pu}{u}_X=0$ and it follows that $u$ is a non-zero constant. Thus $1\in \DD_{0}^{p}$ and \ref{S1} follows from
Proposition \ref{prop:D=D_{0}^{p}}.

\noindent \ref{S1} $\implies$ \ref{S6}: If the graph is $p$-parabolic, again by Proposition \ref{prop:D=D_{0}^{p}}, the function $1$ is a non-zero harmonic function in $\DD_{0}^{p}=\DD^p$.
\end{proof}
Clearly, all statements about superharmonic functions can be equivalently reformulated for subharmonic functions.

\subsection{Kelvin-Nevanlinna-Royden-type characterization}\label{sec:KNR}

Before we state the main result of this subsection, we need to introduce flows on graphs.

Let $V\sse X$. A function $F\colon V\times V \to \RR$ is called \emph{flow} on $V$ if it is skew-symmetric, i.e.,
\[ F(x,y)=-F(y,x), \qquad x,y\in V.\]
The \emph{divergence} of a flow $F$ on $X$ is then defined by
\[ \Div_b F(x):=\frac{1}{m(x)} \sum_{y\in X}b(x,y)F(x,y), \qquad x\in X,\]
provided the sum on the right hand side converges absolutely. An easy application of H\"older's inequality shows that this is so whenever
$F\in \ell^q(X\times X, b)$ for some $q\in [1,\infty]$.

For flows, we have the following version of Green's formula.

\begin{lemma}[Green's formula for flows]\label{lem:GreenFlow}
	Let $F\in \ell^q(X\times X, b)$ be a flow for some $q\in [1,\infty]$. Assume either
	\begin{enumerate}[label=(\alph*)]
		\item\label{lem:GreenFlow0} $g \in C_c(X)$, or
		\item\label{lem:GreenFlow1} $g\in \DD_0^p\cap \ell^\infty(X)$ and $\Div_bF\in \ell^1(X,m)$.
	\end{enumerate}
	{Then the  sums defining the quantities in the equation below} converge absolutely and
	\[\frac{1}{2}\sum_{x,y\in X}b(x,y)F(x,y)(\nabla_{x,y}g)= \sum_{x\in X}\Div_bF(x)g(x)m(x).\]
\end{lemma}
\begin{proof}
\ref{lem:GreenFlow0}: Since $g$ is finitely supported the above sums are finite.  {Since}  $F$ is a flow we compute
	\begin{align*}
	\sum_{x,y\in X}b(x,y)F(x,y)(\nabla_{x,y}g) &= \sum_{x,y \in X} b(x,y)F(x,y)g(x) + \sum_{x,y\in X}b(x,y)F(y,x)g(y) \\
	&=2 \cdot \sum_{x\in X}\Div_bF(x)g(x)m(x).
	\end{align*}

\ref{lem:GreenFlow1}: Let $(g_n)$ be a sequence in $C_c(X)$ such that $g_n\to g\in \DD_0^p\cap \ell^\infty(X)$ and $\abs{g_n(x)}\leq \abs{g(x)}$ for all $x\in X$. By \ref{lem:GreenFlow0} we get
		\begin{align*}
	\frac{1}{2}\sum_{x,y\in X}b(x,y)F(x,y)(\nabla_{x,y}g_n)= \sum_{x\in X}\Div_bF(x)g_n(x)m(x).
\end{align*}
Since $\Div_bF\in \ell^1(X,m)$ and $\abs{g_n} \nearrow \abs{g}$, we can use dominated  convergence on the right-hand side. Hence, also the left-hand side converges.
\end{proof}
	
	The following is a generalisation of \cite{Prado} for the standard Laplacian on locally finite graphs to weighted Laplacians on locally summable graphs. We also note that the proofs differ. Moreover, on manifolds such a result is well-known, see e.g. \cite{PS14}. For $p>1$ such a results was first proven on manifolds in \cite{GT99} (see also \cite{GM}), and for $p=2$, it goes back to \cite{LS84}. For a version of this result in the linear case on graphs, see \cite[Section~6.5]{KLW21} and also \cite{HKLMS}. Note that also $p=1$ would be possible.
	
\begin{theorem}[Kelvin-Nevanlinna-Royden characterization]\label{KNR}
		The following are equivalent:
		\begin{enumerate}[label=(\roman*)]
			\item\label{prop:KNR1} The graph $G=(X,b,m)$ is $p$-parabolic.
			\item\label{prop:KNR2} For all flows $F\in \ell^q(X\times X, b)$, $1/q+1/p=1$, and all (some) measures $m$ such that $(\Div_bF)_+\in \ell^1(X,m)$ or $(\Div_bF)_-\in \ell^1(X,m)$, we have
			\[ \sum_{x\in X}\Div_b F(x)m(x)=0.\]
			\item\label{prop:KNR3} 
			For all flows $F\in \ell^q(X\times X, b)$, $1/q+1/p=1$, and all functions $g\in D^p\cap \ell^\infty(X)$ and all (some) measures $m$ such that $\Div_bF\in \ell^1(X,m)$, we have
			\[ \sum_{x\in X}\Div_b F(x)g(x)m(x)=\frac{1}{2} \sum_{x,y\in X}b(x,y)F(x,y)\nabla_{x,y}g.\]
		\end{enumerate}
\end{theorem}
\begin{proof}
	\ref{prop:KNR1} $\implies$ \ref{prop:KNR2}:  
    Since the graph is $p$-parabolic, by Proposition~\ref{prop:D=D_{0}^{p}}, there exists a sequence $(\phi_n)$ in $C_c(X)$ such that $0\leq \phi_n \leq 1$, and $\phi_n\to 1$ as $n\to \infty$ and $\E_p(\phi_n)\to 0$.
	
Then, by Lemma~\ref{lem:GreenFlow}~\ref{lem:GreenFlow0}, 
		\begin{multline*}
	\frac{1}{2}\sum_{x,y\in X}b(x,y)F(x,y)(\nabla_{x,y}\phi_n)= \sum_{x\in X}\Div_bF(x)\phi_n(x)m(x)\\
	= \sum_{x\in X}(\Div_bF(x))_+\phi_n(x)m(x)- \sum_{x\in X}(\Div_bF(x))_-\phi_n(x)m(x).
	\end{multline*}
	Now, as $n\to \infty$, the left-hand side vanishes since 
	\[\norm{F\nabla \phi_n}_{1,b}\leq  \norm{F}_{q,b} \norm{\nabla \phi_n }_{p,b}\to 0. \]
	
	Since $(\Div_bF)_+\in \ell^1(X,m)$ or $(\Div_bF)_-\in \ell^1(X,m)$ and $\phi_n \nearrow 1$, we can use dominated and monotone convergence on the right-hand side. Thus, using $\phi_n \nearrow 1$, we get
	\[\sum_{x\in X}\Div_bF(x)m(x)= \lim_{n\to \infty}\frac{1}{2}\sum_{x,y\in X}b(x,y)F(x,y)(\nabla_{x,y}\phi_n)=0.\]

\noindent \ref{prop:KNR1} $\implies$ \ref{prop:KNR3}: By \ref{prop:KNR1}, we can apply again Proposition~\ref{prop:D=D_{0}^{p}} and get $D^p=D_0^p$. Then, applying Lemma~\ref{lem:GreenFlow}~\ref{lem:GreenFlow0} to $g\in D^p\cap \ell^\infty(X)$, the result follows.
	
\noindent \ref{prop:KNR2}, \ref{prop:KNR3} $\implies$ \ref{prop:KNR1}: Assume that the graph is $p$-hyperbolic, fix $o\in X$, and set $\Delta:=\Delta_{p,b,m}$ for some measure $m$ with $m(o)>0$. Let $u\in D_{0}^{p}$ be the $p$-harmonic potential at $o$ constructed in Proposition~\ref{existence potential}, so that $\Delta u (x) =\mathds{1}_o(x)$ on $X$ and $\E_p(u)=\cc_p(o)$.
We set $F(x,y):=\p{\nabla_{x,y}u}$ and $g=\mathds{1}_X\in D^p\cap\ell^\infty(X)$. Then $F\in \ell^q(X\times X, b)$, and $\Div_b F= \Delta u$. Thus,
\[\sum_{x\in X}\Div_bF(x)m(x)= \sum_{x\in X}\Delta u(x)m(x)=m(o)>0.\qedhere\]
\end{proof}

\subsection{Poincaré-type characterization}\label{s:Poincare}

This is another generalization of a result in \cite{Prado}. Here, the proof would also go through for arbitrary potentials as long as the functional is non-negative and also $p=1$ is possible. The  inequality  {\color{blue} in \eqref{prop:PI2}} is sometimes called Friedrichs or Poincaré inequality. It is related to Proposition~\ref{prop:D=D_{0}^{p}}~\ref{D8}.

\begin{proposition}\label{Poincaré-type characterization}[Poincaré-type characterization]		
		The following are equivalent:
		\begin{enumerate}[label=(\roman*)]
			\item\label{prop:PI1} The graph $G=(X,b,m)$ is $p$-hyperbolic.
			\item\label{prop:PI2} For all finite $K\sse X$ there is a constant $C(K,p)$ such that
			\[\sum_{x\in K}\abs{\phi(x)}^p\leq C(K,p)\, \E_p(\phi), \qquad \phi\in C_c(X).\]
			\item\label{prop:PI3} For all  $o\in X$ there is a constant $C(o,p)$ such that
			\[\abs{\phi(o)}^p\leq C(o,p)\, \E_p(\phi), \qquad \phi\in C_c(X).\]
		\end{enumerate}
\end{proposition}
\begin{proof}
\ref{prop:PI2} $\iff$ \ref{prop:PI3}: Clearly, if the Poincaré-type inequality holds for all finite subsets, then it also holds for all singletons. On the other side, if it holds for all singletons, then $C(K,p):=\sum_{x\in K}C(x,p)$ is a possible constant for a finite set $K\sse X$.

\noindent \ref{prop:PI3} $\implies$ \ref{prop:PI1}: Assume  that the graph is $p$-parabolic. By Theorem~\ref{thm:criticalNew}, this is equivalent to the existence of a null-sequence $(\phi_n)$, i.e. a sequence in $C_c(X)$ such that $\phi_n(o)=1$ for some fixed vertex $o\in X$, and $\E_p(\phi_n)\to 0$ as $n\to \infty$ and clearly \ref{prop:PI3} cannot hold.

\noindent \ref{prop:PI1} $\implies$ \ref{prop:PI3}: Assume  that \ref{prop:PI3} does not hold, i.e., there exists a vertex $o\in X$ such that for all $k\in \NN$ we find $\phi_k\in C_c(X)$ with
\[\abs{\phi_k(o)}^p> k \, \E_p(\phi_k).\]
Since $\E_p$ is $p$-homogeneous we can assume without loss of generality that $\abs{\phi_k(o)}=1$. Hence, $(\phi_k)$ is a null-sequence for $\E_p$, and therefore $\E_p$ is $p$-parabolic.
\end{proof}

\subsection{Ahlfors-type characterization}\label{sec:A}

Here, we show a maximum principle-type characterization of  $p$-parabolicity which goes back to Ahlfors \cite{Ahlfors}, where a corresponding version was proven for Riemann surfaces and $p=2$. For $p\neq 2$ and Riemannian manifolds, the corresponding result can be found in \cite{PS14, PST14}. For historical reasons, we state the characterization in terms of subharmonic functions.

Before stating our result, we need a lemma which is also used in Subsection~\ref{sec:WMP}. It is the discrete version of the result stating that the pointwise maximum of subsolutions is a subsolution. The proof, which we reproduce for completeness, follows from direct computations. See also \cite[Lemma~4.4]{F:Opti} for a version valid for an arbitrary family of subsolutions. We mention that an arbitrary potential could be added, and also $p=1$ could be included.

\begin{lemma}[$\vee$-stable]\label{lem:vee}
	Let  $V\sse X$. Let $u_i \in \FF^p(V)$, $i=1,2$ 
satisfy $\Delta_pu_i (x)\leq f(x) \p{u_i(x)} $ on $V$ for some  function $f\in C(X)$. Then, the  pointwise maximum $w:=u_1\vee u_2$ is also in $\FF^p(V)$ and satisfies $\Delta_pw(x)\leq f(x) \p{w(x)}$ on $V$.
\end{lemma}
\begin{proof} Since $|u_1\vee u_2|\leq |u_1|+|u_2|$, it follows from the $p$-triangle inequality that $w\in \FF^p(V)$.
	Fix $x\in V$. Without loss of generality, we may assume that $u_1(x)\geq u_2(x)$. Then,
	\[\nabla_{x,y}w = u_1(x)-w(y)\leq \nabla_{x,y}u_1.\]
	Since $\p{\cdot}\colon \RR \to \RR$ is monotone increasing, we get
	\[\Delta_{p}w(x)\leq \Delta_{p}u_1(x)\leq f(x)\p{u_1(x)} =f(x)\p{w(x)},\]
	i.e., $\Delta_pw(x) \leq f(x) \p{w(x)}$ on $V$.
\end{proof}

\begin{proposition}[Ahlfors-type characterization]	\label{Ahlfors}	
		The following are equivalent:
		\begin{enumerate}[label=(\roman*)]
			\item\label{A1} The graph $G=(X,b,m)$ is $p$-parabolic.
			\item\label{A2} For all proper subsets $V\subsetneq X$ (i.e., $\partial_e V \neq \emptyset$), and all functions $u$ which are  bounded   on $\overline{V}=V\cup \partial_e V$ and $p$-subharmonic  on $V$, we have
			\[\sup_{\overline{V}} u = \sup_{ \partial_e V}u.\]
		\end{enumerate}
\end{proposition}
\begin{proof}
\ref{A2} $\implies$ \ref{A1}: Let the graph be $p$-hyperbolic. By Proposition~\ref{prop:S}, there exists a non-constant $p$-subharmonic function $u$ which is bounded from above. Hence, there are $x_1, x_2\in X$ such that $u(x_1)\neq u(x_2)$. Set $\gamma := (u(x_1)+u(x_2))/{2}$ and let $V:=\set{x\in X : u(x)> \gamma}$. Then $V$ is non-empty and
\[ \sup_{\partial_e V} u \leq \gamma < \sup_{\overline{V}}u.\]

\noindent \ref{A1} $\implies$ \ref{A2}: Let the graph be $p$-parabolic and let $V\subsetneq X$. Suppose that $u$ is bounded on $\overline{V}$ and $p$-subharmonic on $V$. We show that the additional assumption $\sup_{\partial_e V} u < \sup_{\overline{V}}u$ results in a contradiction.

Indeed, by the additional assumption, there exists $\epsilon > 0$ such that $\max_V u> \sup_{\partial_e V} u +\epsilon$. Set
\begin{align*}
u_\epsilon:= \begin{cases}
u \vee (\sup_{\partial_e V} u +\epsilon) \qquad &\text{in } V \\
\sup_{\partial_e V} u +\epsilon \qquad &\text{in } X\setminus V.
\end{cases}
\end{align*}
By Lemma \ref{lem:vee}, $u_\epsilon$ is $p$-subharmonic in $V$, and it is obviously harmonic in $X\setminus \overline{V}$. Furthermore, for all $x\in \partial_e V$ we have
\[
m(x)\Delta_p u_\epsilon(x)=\sum_{y\in V}b(x,y)\p{\sup_{\partial_e V} u + \epsilon - (u(y)\vee (\sup_{\partial_e V} u + \epsilon)) } \leq 0.
\]
Hence, $u_\epsilon$ is subharmonic and bounded on $X$. By Proposition~\ref{prop:S}, $u_\epsilon$ is constant which contradicts the additional assumption and finishes the proof.
\end{proof}

\subsection{A Lenz-Puchert-Schmidt-type characterization}\label{sec:LPS}
Here, we show a generalization to $p\in (1,\infty)$ of a characterization of ($2$-)parabolicity very recently obtained in \cite{LPS23}. The first basic idea is to connect parabolicity not only to the vanishing of the capacity of singletons (or, equivalently, finite sets) in the graph but also to the vanishing of a suitably defined capacity of subsets  of the boundary of the graph.

The second basic idea is to show that this is also equivalent to a statement closely connected to the Khas'miniski\u{\i}-type characterization in terms of finite energy $p$-superharmonic functions which will also play a role in Subsection~\ref{sec:K}.

\subsubsection{Boundaries, limits, and capacities}

Let $Y$ be a compact topological Hausdorff space. Then $Y$ is called \emph{compactification} of $X$, if $Y$ contains a copy of $X$, the restriction of the topology of $Y$ on $X$ is the discrete topology, and $X$ is dense in $Y$. The \emph{boundary} $\partial_YX$ of $x$ in $Y$ is then defined by $\partial_YX= Y\setminus X$.

A \emph{pseudometric} on $X$ is a function $\sigma\colon X\times X \to [0,\infty)$  which  is symmetric and satisfies the triangle inequality. We write $x \simeq y$ if $\sigma (x,y)=0$ and we extend this equivalence relation to sequences in $X$ by saying that  two  sequences $(x_n)$ and $(y_n)$ are $\sigma$-equivalent in $X$ if $\lim_{n\to\infty}\sigma (x_n,y_n)=0$. The \emph{completion} $\overline{X}^\sigma$ of $X$ with respect to $\sigma$ is the set of all equivalence classes of $\sigma$-Cauchy sequences.  The \emph{boundary} is then defined as $\partial_{\sigma}X = \overline{X}^\sigma \setminus (X/ \simeq)$.

If $\sigma$ is a metric inducing the discrete topology on $X$, then
\[ \partial_{\sigma}X= \partial_{\overline{X}^\sigma} X. \]

Let $\sigma$ be a pseudometric. Then, for all $r\geq 0$ and $x\in X$, we set
\[B_{r,\sigma}(x):= \set{y\in X : \sigma(x,y)\leq r}.\]

If $\sigma$ is a metric on $X$ that induces the discrete topology and $f\in C(X)$, then we set for all $A\sse \overline{X}^\sigma$,
\[ \liminf_{x\to A}f(x):= \sup_{A\sse O \text{ open in } \overline{X}^\sigma 
} \inf_{x\in O\cap X} f(x),\]
and
\[ \liminf_{x\to \infty}f(x):= \sup_{K\sse X \text{ finite }} \inf_{x\in X\setminus K} f(x).\]

We note that $\liminf_{x\to \partial_\sigma X}f(x)\geq  \liminf_{x\to \infty}f(x)$ and that equality holds if $\overline{X}^\sigma$ is compact (see \cite[Proposition~3.2]{LPS23}).

Now we turn to capacities, and extend the standard definition to the boundary of $X$.

Let $(X,b,m)$ be a graph. Then the \emph{$(o,p)$-capacity} of $V\sse X$ with respect to $\E_p$ and $o\in X$ is defined by
\[
\cc_{o,p}(V):= \inf \{\norm{u}_{o,p}\,:\, u\in \DD^p,\, u  \geq 1 \text{ on } V\},
\]
where we use the convention that the infimum over the empty set is $\infty$.

The normal contraction $0\vee \cdot \wedge 1$ can be used to show that the infimum can be taken over all $u\in \DD$ such that $0\leq u \leq 1$ on $X$ and $u = 1$ on $V$.

Let $Y$ be a compactification of $X$. Then, we set for all $A\sse Y$,
\[ \cc_{o,p}(A):= \inf\{\cc_{o,p} (O\cap X)\,:\, O \text{ is open in } Y, \, A\sse O\}.\]

Note that both definitions coincide on subsets of  $X$, and thus using the same notation does not lead to confusions. Moreover, we note that $\cc_{o,p}(A)=0$ if and only if there exists a sequence $(O_n)$ of open sets in $Y$ with $A\sse O_n$ and a sequence $(f_n)$ in $\DD^p$ such that $f_n\geq 1_{O_n\cap X}$, and $\norm{f_n}_{o,p}\to 0$ as $n\to \infty$.

\subsubsection{Small perturbations}

The following lemma is needed in the main result to go from a pseudometric to a metric. It is the generalisation of \cite[Proposition~2.6]{LPS23} from $p=2$ to $p>1$. We point out that the proof of part (a) is different from the original one where the bilinearity of the form $\E_2$ was used. We use Minkowski's inequality instead. Moreover, note that non-negative potentials and $p=1$ could be included.

\begin{lemma}[Small perturbations]\label{lem:sp}
	We have the following:
	\begin{enumerate}[label=(\alph*)]
	\item\label{lem:spA} For all $\epsilon >0$ there exists a strictly positive function $f_\epsilon\in C(X)$ such that for any $g\in C(X)$ with $g(x)\in (-f_\epsilon(x),f_\epsilon(x))$, $x\in X$, we have
	\[ \E_p(g)<\epsilon.\]
	\item\label{lem:spB} For all $\epsilon > 0$ and $u\in \DD^p$, there exists a function $u_\epsilon \in \DD^p$ such that $u_\epsilon (x) \neq u_\epsilon (y)$ for all $x,y\in X$ with $x\neq y$, $\norm{u-u_\epsilon}_{\infty}<\epsilon$, and
	\[\abs{\E_p(u-u_\epsilon)}< \epsilon.\]
	\end{enumerate}
\end{lemma}
\begin{proof}
\ref{lem:spA}: Having fixed $\epsilon > 0$, 
choose $f_\epsilon > 0$ such that
	\[  \sum_{x\in X}f_\epsilon(x)\E_p(\mathds{1}_x)^{1/p}< \epsilon^{1/p}.\]	
Now, let $g$ be such that $g(x)\in (-f_\epsilon(x),f_\epsilon(x))$, $x\in X$. Let $(K_n)$ be an increasing exhaustion of $X$ with finite sets, and set
	\[ g_n:= g\,\mathds{1}_{K_n}.\]
	Then, $(g_n)$ converges pointwise to $g$.	Moreover, $\E_p^{1/p}$ is a semi-norm on $\DD^p$, where the triangle inequality follows from Minkowski's inequality. Hence, using Fatou's lemma (i.e., the semi-continuity of $\E_p$) and the triangle inequality with respect to $\E_p^{1/p}$, we get
	\begin{multline*}
		\E_p^{1/p}(g)\leq \liminf_{n\to \infty}\E_p^{1/p}(g_n)\leq \liminf_{n\to \infty} \sum_{x\in K_n}\abs{g(x)}\E_p^{1/p}(\mathds{1}_x) \leq \sum_{x\in X}\abs{f_\epsilon(x)}\E_p^{1/p}(\mathds{1}_x)
		\\< \epsilon^{1/p} .
	\end{multline*}
	This shows the first claim.
	
\noindent \ref{lem:spB}: This follows now from \ref{lem:spA} in the same way as in \cite{LPS23}. For convenience, we show the nice argument here as well. Let $\epsilon > 0$ and take $f_\epsilon$ as in \ref{lem:spA} such that $f_\epsilon(x)< \epsilon$. Let $F\in C(X)$ be a function such that $F(x)-F(y)$ is irrational for any $x,y\in X, x\neq y$. Moreover, let $g\in C(X)$ be such that $g(x)\in (-f_\epsilon(x),f_\epsilon(x))$ and $u(x)-F(x)-g(x)$ is rational for all $x\in X$. Define $u_\epsilon (x):= u(x)-g(x)$, $x\in X$. Then, $u_\epsilon \in \DD^p$ by the triangle inequality. Furthermore, $\norm{u-u_\epsilon}_{\infty}=\norm{g}_{\infty}< \epsilon$ as well as $\E_p(u-u_\epsilon)=\E_p(g)<\epsilon$.
	
	For all $x\neq y$, we have
	\[\nabla_{x,y}u_\epsilon= \nabla_{x,y}F + \nabla_{x,y}(u-F-g).\]
	The right-hand side above is a sum of a irrational and a rational number, and thus, $\nabla_{x,y}u_\epsilon \neq 0$.	
\end{proof}

\subsubsection{Intrinsic and path pseudometrics}

Two special pseudometrics are of particular interest for us: path pseudometrics and intrinsic metrics. Here, we recall some basic definitions and statements. We want to mention that we are aware of another generalisation of the notion of intrinsic metrics for $p\neq 2$ in \cite{KM16}. However, we do a different generalisation of ($p=2$)-case here, which seems to be more natural in our setting. Path pseudometrics are discussed in detail in \cite[Section~11.2]{KLW21} and in \cite{KM19}, compare  also with \cite{LPS23}.

A non-negative symmetric function $w$ on $X\times X$ is called \emph{$p$-intrinsic} with respect to the edge weight $b$ and the measure $0\leq m\in C(X)$ if, for all $x\in X$, we have
\[\sum_{x\in X}b(x,y)w^p(x,y)\leq m(x).\]
We will be mostly  interested in $p$-intrinsic functions which are also pseudometrics.

An infinite sequence $(x_n)$ of pairwise different vertices such that $x_n\sim x_{n+1}$ for all $n\in \NN$ is called \emph{ray} on $X$. For a symmetric function $w\colon X\times X\to [0,\infty)$ we define the \emph{length} $l_w(\gamma)$ of the path $\gamma=(x_i)_{i=1}^n$ with respect to $w$ via
\[ l_w(\gamma)=\sum_{i=1}^{n-1} w(x_i,x_{i+1}),\]
and the length of the ray $\gamma=(x_i)_{i=1}^{\infty}
$ is given by $ l_w(\gamma)=\sum_{i=1}^{\infty} w(x_i,x_{i+1}).$

The \emph{path pseudometric} $d_w\colon X\times X\to [0,\infty]$ with respect to $w$ on $X$ is given by
\[
d_w(x,y):= \inf\{ l_w(\gamma) \,:\, \gamma \text{ is a path or ray from } x \text{ to } y\}.
\]
Moreover, we call a symmetric function $w\colon X\times X\to [0,\infty )$ \emph{alternative edge weight} if $w(x,y)> 0$ whenever $x\sim y$, $x,y\in X$.

There is a simple condition under which a symmetric function induces a discrete path metric.
\begin{lemma}
\label{lem:KM19}
Let $w$ be a non-negative alternative edge-weight on $X\times X$ such that for every $x$ there exists $r_x>0$ such that the set $\{y\sim x\,:\, w(x,y)< r_x\}$ is finite, then $d_w$ is a metric and  $(X, d_w)$ is discrete.
\end{lemma}
\begin{proof}
The assumptions imply that $\delta_x=\inf_{y\sim x} w(x,y)>0$ so that  $d_w(x,x')\geq \delta_x$ for every $x'\ne x$ and the conclusion follows at once.
\end{proof}

\subsubsection{The characterization and a corollary}

We are now in a position to show the main result of this section, Theorem~\ref{prop:LPS}, which is a generalisation of \cite[Theorem~4.2]{LPS23} from $p=2$ to $p>1$. For $p=2$, parts of it appeared already in \cite{Schm1}. Moreover, we want to mention that in the continuum, the implication ``\ref{L1} $\implies$ \ref{L2}'' is well-known for general $p$ and appears e.g. in \cite{Valtorta}. The proof of Theorem~\ref{prop:LPS} goes along the lines of its linear version in \cite{LPS23}. Since it is a very recent result, we show it here in detail. Note that $p=1$ could be included.

\begin{theorem}[Lenz-Puchert-Schmidt-type characterization]\label{prop:LPS}
	The following are equivalent:
		\begin{enumerate}[label=(\roman*)]
			\item\label{L1} The graph $G=(X,b,m)$ is $p$-parabolic.
			\item\label{L4} For one (or, equivalently, every) compactification $Y$ of $X$, we have \[\cc_{o,p}(\partial_Y X)=0.\]
			\item\label{L2} There is a function $f\in \DD^p$ such that $f(x) \to \infty$ as $x\to \infty$, i.e.,  for every $ M > 0 $  there exists $F \subseteq X $ finite such that
			$f > M $ in $ X \setminus F $.
			\item\label{L3} There is a $p$-intrinsic metric $\sigma$ that induces the discrete topology on $X$ such that distance balls with respect to $\sigma$ are finite.
			 \item\label{L3'} There exists a $p$-intrinsic alternative edge weight $w$ with respect to some finite measure such that the distance balls with respect to $d_w$ are finite.
		\end{enumerate}
\end{theorem}
\begin{proof}
\ref{L1} $\implies$ \ref{L4}: We can use Proposition~\ref{prop:D=D_{0}^{p}} and get the existence of an increasing null-sequence $(e_n)$ in $C_c(X)$ such that $0\leq e_n \nearrow 1$. Consider $f_n:=1-e_n$, $n\in \NN$. Then $0\leq f_n\leq 1$ and $f_n=1$ outside of a  finite set. Hence,
	\[ 0\leq \cc^p_{o,p}(\partial_Y X) \leq \E_p(f_n)+\abs{f_n(o)}^p= \E_p(e_n)+\abs{1-e_n(o)}^p\to 0,\qquad n\to \infty.\]
	
\noindent \ref{L4} $\implies$ \ref{L2}: By assumption, there exists a sequence of open sets in $Y$ $(O_n)$ with $\partial_YX\sse O_n$ and a sequence of functions $(f_n)$ such that $f_n\geq \mathds{1}_{O_n\cap X}$ for which
	\[ \lim_{n\to \infty}\norm{f_n}_{o,p}\to 0. \]
	
By passing  to a subsequence, which we also denote by $(f_n)$, we can assume that 	$\sum_{n\in\NN}\norm{f_n}_{o,p}<\infty$. Thus, $f:=\sum_{n\in\NN}f_n$ converges in the Banach space $(\DD^{p},\norm{\cdot}_{o,p})$, and therefore $f\in \DD^p$. By construction, $f\geq n$ on $U_n:=\cap_{k=1}^n O_k$. Since $U_n$ is open and $\partial_YX\sse U_n$, $\liminf_{x\to \partial_YX}f(x)\geq n$. Taking the limit yields the result.
	
\noindent \ref{L2} $\implies$ \ref{L3}: By Lemma~\ref{lem:sp}, we can assume that the function $f$ in \ref{L2} is injective. Hence $\sigma_f\colon X\times X\to [0,\infty)$ defined by $\sigma_f(x,y):=\abs{\nabla_{x,y}f}$ is an intrinsic metric with respect to the measure $0\leq m_f\in C(X)$ defined by
	\[m_f(x):= \sum_{x\in X}b(x,y)\sigma^p_f(x,y).\]
	Note that $m_f$ is a finite measure since $m_f(X)=2\E_p(f)<\infty$. Since $\liminf_{x\to \infty}f(x)=\infty$, $B_{r,\sigma}(o)$ is a finite set for all $o\in X$ and $r\geq 0$. Moreover, by the choice of $f$, $\sigma_f$ induces the discrete topology on $X$.
	
\noindent \ref{L3} $\implies$ \ref{L1}: Let $\sigma$ be a $p$-intrinsic metric with respect to the measure $0\leq m\in C(X)$ which induces the discrete topology and such that $B_{r,\sigma}(o)$ is a finite set for all $r\ge 0$ and $o\in X$. Let $(K_n)$ be an increasing exhaustion of $X$ by finite sets. Let us consider the functions $\sigma_n\colon X\to [0,\infty)$ defined via
	\[ \sigma_n (x):= \inf_{y\in K_n}\sigma (x,y),\]
and let $e_n:= 0\vee (1-\sigma_n) \geq 0$. Since $K_n$ is finite and distance balls are finite, $e_n\in C_c(X)$. We are left to show that $(e_n)$ is a null-sequence. Firstly, note that $\abs{\nabla_{x,y}\sigma_n}\leq \sigma(x,y)$ for all $x,y\in X$. Since $0\vee \cdot $ is a normal contraction we have
\begin{align*}
0\leq \E_p(e_n)\leq \E_p(\sigma_n)&= \frac{1}{2}\sum_{(x,y)\in (X\times X)\setminus (K_n\times K_n)} b(x,y) \abs{\nabla_{x,y}\sigma_n}^p\\
 &\leq\frac{1}{2}\sum_{(x,y)\in (X\times X)\setminus (K_n\times K_n)} b(x,y) \abs{\sigma(x,y)}^p
 \\
 &\leq m(X\setminus K_n)\to 0, \quad n\to \infty.
\end{align*}
Hence, $(e_n)$ is a null-sequence and therefore, the graph is $p$-parabolic.

\noindent \ref{L3} $\implies$ \ref{L3'}: According to \cite[Lemma 2.7]{LPS23}, since $\sigma$ is a metric,  $d_\sigma\geq \sigma$  and $d_\sigma(x,y)=\sigma(x,y)$ if $x\sim y$ so distance balls with respect to $d_\sigma$ are contained in the corresponding balls with respect to $\sigma$ and $\sigma$ is obviously a $p$-intrinsic weight, so we can take $w=\sigma$.

\noindent \ref{L3'} $\implies$ \ref{L3}: Since $d_w(x,y)\leq w(x,y)$ for all $x,y\in X$ such that $x\sim y$, also $d_w$ is $p$-intrinsic. Moreover, it is a pseudometric. If we can show that it induces the discrete topology it follows that $d_w$ is actually a metric.  This can be seen as follows: Since for all $r\geq 0$ and $o\in X$, the balls $ B_{r,d_w}(o)$ are finite,  so are also the sets $\{y\sim x\,:\, w(x,y)< r\}$. By Lemma~\ref{lem:KM19}, $d_w$ is a metric and $(X,d_w)$ is discrete.
\end{proof}

We also get the following nice characterization on locally finite graphs as a consequence of the previous proposition. It was observed first in \cite{Y77}, see also \cite[Corollary~4.3]{LPS23}.

Before we state the result, we need one more definition: The set of rays $\Gamma$ on $X$ with respect to the graph $b$ is called \emph{$p$-null} if there exists an alternative edge weight $w$ such that $\sum_{x,y\in X}b(x,y)w^p(x,y)< \infty$ and $l_w(\gamma)=\infty $ for all rays $\gamma \in \Gamma$.

\begin{corollary}[Yamasaki-type characterization]
	We have the following:
	\begin{enumerate}[label=(\alph*)]
			\item\label{Y1} If the graph is $p$-parabolic then the set of rays is $p$-null.
			\item\label{Y2} A locally finite graph is $p$-parabolic if and only if the set of rays is $p$-null.
		\end{enumerate}
\end{corollary}
\begin{proof}
	Mutatis mutandis, this follows as in \cite[Corollary~4.3]{LPS23}.
\end{proof}
Note that also here $p=1$ can be included.

\subsection{The Area Perspective}\label{s:area}

	Before we state our main result, we need some definitions. By $d$ we denote the \emph{combinatorial graph metric}, i.e., for all $x,y\in X$ the number $d(x,y)\in \NN$ denotes the least number of edges of paths connecting $x$ and $y$. 
	Furthermore, for $V\sse X$ and $x\in X$, we set $d(V,x):=\min_{y\in V}d(y,x)$, and
	\begin{align*}
		B_r(V):=\set{x\in X : d(V,x) \leq r}, \\
		S_r(V):=\set{x\in X : d(V,x) = r}.
	\end{align*}
	
	A function $f$ is called \emph{spherically symmetric} with respect to $V$ if there exists a function $g \colon \NN \to \RR $ such that $f(x)=g(r)$ for all $x\in S_r(V)$ and $r\in \NN$. To simplify the notation, we will also write $f(r)=f(x)$ for all $x\in S_r(V)$ and $r\in \NN$.	
		
	A graph with potential $c$ is called \emph{model graph} on $V$ if the outer and inner weighted degrees with respect to $V$, defined respectively by
	\[
	k_\pm(x)=m(x)^{-1}\sum_{y\sim x\atop d(V,y)=d(V,x)\pm 1} b(x,y),
	\]
	and the weighted potential are spherically symmetric with respect to $V$, see Fig.~\ref{Fig1}.
	We will (mostly) consider the case where $V=\{x_0\}$, and write for simplicity $B_r(x_0)$ and $ S_r(x_0)$, possibly omitting $ x_0 $ and writing $ B_r $, $ S_r $ when the reference point is fixed. 	
\begin{figure}[h]\centering			
		\begin{tikzpicture}[
			scale=.7,
			every node/.style={circle, draw, fill=gray!30, minimum size=1mm},
			invisible/.style={draw=none, fill=none}]
			
			\node (o) at (0,0) {$x_0$};
			
			\node (v1) at (2,1) {};
			\node (v2) at (2,-1) {};
			\node (v3) at (1.5,2.2) {};
			\node (v4) at (1.5,-2.2) {};
			
			\node (v5) at (4,2) {};
			\node (v6) at (4,0.75) {};
			\node (v7) at (4,-0.75) {};
			\node (v8) at (4,-2) {};
			
			\node (v9) at (6.5,-0.5) {$x$};
			\node (v10) at (6,1.5) {};
			
			\node (v11) at (8,1.75) {};
			\node (v12) at (8,-2) {};   		
			\node[invisible] (v13) at (8,3) {$\ldots$};
			
			\draw (o) -- (v1);
			\draw (o) -- (v2);
			\draw (o) -- (v3);
			\draw (o) -- (v4);
			
			\draw (v1) -- (v2);
			
			\draw[red, ultra thick] (v1) -- (v5);
			\draw[red, ultra thick] (v1) -- (v6);
			\draw[red, ultra thick] (v2) -- (v7);
			\draw[red, ultra thick] (v2) -- (v8);
			\draw[red, ultra thick] (v3) -- (v5);
			\draw[red, ultra thick] (v4) -- (v8);
			
			\draw (v5) -- (v6);
			\draw (v7) -- (v8);
			\draw[blue, ultra thick] (v6) -- (v9);
			\draw[blue, ultra thick] (v7) -- (v9);
			
			\draw (v5) -- (v10);
			\draw (v9) -- (v10);
			
			\draw[brown, ultra thick] (v9) -- (v11);
			\draw[brown, ultra thick] (v9) -- (v12);
			
			\draw[dashed] (v11) -- (v13);
			\draw[dashed] (v10) -- (v13);
		\end{tikzpicture}	
		\caption{Example of \textcolor{blue}{\(k_-(x)\)}, \textcolor{brown}{\(k_+(x)\)},  \textcolor{red}{\(\partial B(1)\)} w.r.t. $x_0$}
		\label{Fig1}
\end{figure}

In this case, for a spherically symmetric function $ f $ on a model graph we have the following formulas for the $p$-Laplacian:
\[
\Delta_pf(0) = k_+(0) \left( f(0)-f(1) \right)^{\langle p-1 \rangle}
\]
and, for $ r > 0 $,
\[
\Delta_pf(r) = k_+(r) \left( f(r)-f(r+1) \right)^{\langle p-1 \rangle} + k_-(r) \left( f(r)-f(r-1) \right)^{\langle p-1 \rangle}.
\]
We also use the following boundary notation: for every $ r \geq 0 $ we set
\[
\partial B_{x_0}(r)=\partial B(r) := \sum_{x \in S_r, y \in S_{r+1}} b(x,y).
\]

It follows from the definitions that for model graphs
\begin{equation}
	\label{eq model}
	m(S_r)=\frac{\partial B(r)}{k_+(r)}=\frac{\partial B(r-1)}{k_-(r)}
\end{equation}
and that all quantities are finite for locally finite graphs.

The goal of this subsection is to prove the following result.

\begin{theorem}[Area characterization]\label{Area}		
		Let the graph be locally finite.
		\begin{enumerate}[label=(\alph*)]
			\item\label{Area1}If \[ \sum_{r=1}^\infty \left( \frac{1}{\partial B_{x_0}(r)}\right)^{1/(p-1)}= \infty\] for some $x_0\in X$, then the graph is  $p$-parabolic.
			\item\label{Area2} A \emph{model} graph with respect to $x_0\in X$ is $p$-parabolic if and only if
			\[ \sum_{r=1}^\infty \left(\frac{1}{\partial B_{x_0}(r)}\right)^{1/(p-1)}= \infty.\]
			 Furthermore, if the model graph is $p$-hyperbolic then its Green's function $g$ with respect to $x_0\in X$ is radial and given by
			\[
			g(r)= \sum_{k=r}^{\infty} \left(\frac{1}{\partial B_{x_0}(k)}\right)^{1/(p-1)}, \qquad r\in \NN_0.
			\]
		\end{enumerate}
\end{theorem}
\begin{proof} Set $\partial B_{x_0}(r)=\partial B(r)$ and $S_r(x_0)=S_r$.
	
\ref{Area1}: We prove that, if
\[
\sum_{r=1}^\infty \frac{1}{(\partial B(r))^{1/(p-1)}} = \infty,
\]
then there exists a null-sequence $ (e_n) $ as in Proposition \ref{prop:D=D_{0}^{p}}.

First of all, if  $ \phi$ is a  spherically symmetric function,  by spherical  symmetry $\nabla_{x,y}\phi=0$ if $x,y\in S_r$. 
If, in addition, $ \phi(x) = 0 $ for every $ x \in S_r $, $ r \geq n+1 $ then,  
\begin{align*}
\E_p(\phi_n) &= \frac{1}{2} \sum_{x,y\in X} b(x,y)\abs{\nabla_{x,y}\phi}^p \\
&= \sum_{r=0}^{n} \sum_{x \in S_r \atop y \in S_{r+1}} b(x,y)\abs{\phi(x)-\phi(y)}^p \\
&=  \sum_{r=0}^{n} \sum_{x \in S_r \atop y \in S_{r+1}} b(x,y)\abs{\phi_n(r)-\phi(r+1)}^p \\
&= \sum_{r=0}^{n} \abs{\phi(r)-\phi(r+1)}^p \partial B(r).
\end{align*}
Next,  define $ e_n(x) = e_n(r) $ recursively by $e_n(0)=1$, $e_n(r)=0$ if $r\geq n=1$ and
\[
e_n(r)-e_n(r+1)= \frac{c_n}{(\partial B(r))^{1/(p-1)}}\,  \text{ if }\, 0\leq r\leq n-1,
\]
with
\[
c_n= \left( \sum_{r=0}^n \left(\frac{1}{\partial B(r)}\right)^{1/(p-1)}\right)^{-1}.
\]
Thus, for every $1\leq r\leq n$
\[
e_n(r)=e_n(0)-\sum_{j=0}^{r-1}[e_n(j)-e_n(j+1)]=1 - c_n\sum_{j=0}^{r-1}
\left(1\over \partial B(j)\right)^{1/(p-1)}.
\]
It follows that $ e_n \nearrow 1 $ and, using the previous computation,
\[
\E_p(e_n) = \sum_{r=0}^{n} \partial B(r) \frac{c_n^p}{\left( \partial B(r) \right)^{p/(p-1)}} = c_n^{p-1} \to 0
\]
as $ n \to \infty $. This concludes the proof of \ref{Area1}.

\ref{Area2}: We only need to prove that, if
\[
\sum_{r=1}^\infty \frac{1}{(\partial B (r))^{1/(p-1)}}< \infty,
\]
then the graph has a Green's function.
We  define $ g\colon X \to \mathbb{R}^+ $ such that, for every $ r \geq 0$ and $ x \in S_r $,
\[
g(x)=g(r)= \sum_{k=r}^{\infty} \left(\frac{1}{\partial B(k)}\right)^{1/(p-1)}.
\]
Then $g$ is spherically symmetric with
\[
g(r)-g(r+1)=\frac 1{\partial B(r)^{1/(p-1)}}.
\]
It follows from the expression of the $\Delta_p$ for radially symmetric functions on model graphs and  \eqref{eq model} that
\[
\Delta_pg(0) = \frac{\partial B(0)}{m(0)} \frac{1}{\partial B(0)} =\frac{1}{m(0)}
\]
and, for $ r > 0 $,
\[
\Delta_pg(r) = \frac{k_+(r)}{\partial B(r)} - \frac{k_-(r)}{\partial B(r-1)} = 0,
\]
as required to complete the proof.
\end{proof}

\subsection{Weak Maximum Principle and Liouville Property}\label{sec:WMP}

In this subsection we introduce the weak maximum principle and show its equivalence with the $p$-parabolicity of the graph. Moreover, we give the definition of the Khas'minski\u{\i}-type criterion and prove that it is a sufficient condition for the $p$-parabolicity to hold. In this subsection, $p=1$ could be included.

\begin{definition}
We say that a graph $G=(X,b,m)$ satisfies the weak maximum principle \eqref{W} if for  every nonconstant, bounded above function $u\in \FF^p$ and for every $\gamma < u^*:=\sup u$
\[
\tag{W}\label{W}
\sup_{\Omega_{\gamma}} \Delta_p u > 0
\]
on the superlevel set $ \Omega_{\gamma} := \set{ x\in X \, : \, u(x) > \gamma } $.
\end{definition}
Recall that we write $\kappa(x) \to \infty$ as $x\to \infty$ for a function $\kappa\in C(X)$ if  for every $ M > 0 $  there exists $F \subseteq X $ finite such that $\kappa > M $ in $ X \setminus F $.
\begin{definition}
We say that a graph $G=(X,b,m)$ satisfies the Khas'minski\u{\i} criterion \eqref{K} if there exists a function $\kappa$ satisfying $\kappa(x) \to \infty$ as $x\to \infty$ and
\[
\tag{K}\label{K}
\Delta_p \kappa \geq 0 \text{ on } X \setminus K,
\]
for some finite set $ K $.
\end{definition}

We prove the following proposition.

\begin{proposition}\label{prop:weak_char}
The following holds:
\begin{enumerate}[label=(\alph*)]
\item\label{prop:weak_char1} A graph is $p$-parabolic if and only if \eqref{W} holds. 
\item\label{prop:weak_char2} If
\eqref{K} holds then so does \eqref{W} and therefore the graph is $p$-parabolic.
\end{enumerate}
\end{proposition}

\begin{proof}
\ref{prop:weak_char1}: First, assume that the graph is $p$-hyperbolic, so that by Proposition~\ref{prop:S} there exists a nonconstant, bounded above function $u$ such that $\Delta_p u \leq 0$. Then $\Omega_{\gamma} \neq \emptyset$ for every $ \gamma < u^*=\sup u $ and
\[
\sup_{\Omega_{\gamma}} \Delta_p u \leq 0,
\]
which contradicts \eqref{W}. This shows that \eqref{W} implies $p$-parabolicity.

For the converse, assume that there exists a nonconstant, bounded above function $u$ and $\gamma<u^*$ such that $\sup_{\Omega_{\gamma}} \Delta_p u \leq 0$. Without loss of generality we may assume that $u^*>0$  and  $\gamma=0$. Define  $v:X\to \mathbb{R}$ by
\begin{align*}
v:= u \vee 0 = \begin{cases}
u &\text{ in } \Omega_0 \\
0 &\text{ in } X \setminus \Omega_0.
\end{cases}
\end{align*}
An easy calculation shows that $v$ is subharmonic (or apply Lemma~\ref{lem:vee}), thus negating the $p$-parabolicity.

\noindent \ref{prop:weak_char2}:
Assume that \eqref{K} holds and let $\kappa$ be such that $\kappa(x)\to \infty$ as $x\to\infty$, and $\Delta_p \kappa \geq 0$ on $X\setminus K$ for some finite set $K$. By translating   $\kappa$, if necessary, we may assume that $\min_K \kappa=0$,  so that $\kappa\geq\max_{\partial_e K} \kappa\ge 0$ on $ X \setminus K $ by the weak comparison principle, Lemma~\ref{lem:WCPNonNegativePotential}.

Assume by contradiction that \eqref{W} does not hold, so that there exists a nonconstant, bounded above function $u$ such that $ \sup_{\Omega_{\gamma}} \Delta_p u \leq 0$ for some $\gamma<u^*$. Again by Lemma~\ref{lem:WCPNonNegativePotential}, $u$ does not attain a maximum. We may assume that $\gamma$ is such that $\max_{{K}} u < \gamma < u^*$, so that $ \Omega_{\gamma} \cap {K} = \emptyset $. Let $x_0 \in \Omega_{\gamma}$ such that $\gamma<\gamma'<u(x_0)$ and choose $\epsilon$ such that $\epsilon \kappa(x_0) < \gamma'-\gamma$. We now define the set $A:=\left\{ x\in \Omega_\gamma \, : \, u(x) > \gamma + \epsilon \kappa (x)\right\}$, which is clearly a subset of $\Omega_{\gamma}$ by definition. Note that $x_0 \in A \neq \emptyset$ and that $A$ is finite because of the properties of $\kappa$. Note that the following system is then satisfied
\begin{align*}
\begin{cases}
\Delta_p u \leq 0\leq \Delta_p (\gamma+\epsilon \kappa) & \text{ in } A \\
u \leq \gamma + \epsilon \kappa & \text{ on } \partial_e A.
\end{cases}
\end{align*}
By Lemma \ref{lem:WCPNonNegativePotential}, we conclude that $u\leq \gamma+\epsilon \kappa $ on $A$, which is a contradiction. This completes the proof.
\end{proof}

\subsection{Khas'miniski\u{\i}-type characterization}\label{sec:K}
This subsection is devoted to show a Khas'miniski\u{\i}-type characterization in the quasi-linear setting on graphs. On manifolds a corresponding version is given in \cite{Valtorta}, see also \cite{ValtortaDiss}. The main idea there {is} to use the existence of solutions to the so-called obstacle problem. A similar argument on locally finite graphs has been shown to work in \cite{HKO}. We show the corresponding version for locally summable graphs in Appendix~\ref{sec:Alternative}. Here, we provide a shorter proof that is more directly based on variational methods. This subsection uses a result of Subsection~\ref{sec:LPS}.

\begin{theorem}\label{khasm}
	Let $G=(X,b,m)$ be a $p$-parabolic graph and let $ K \subseteq X $ be a finite set. Then there exists a non-negative function $\kappa\in C(X)$ such that
	\begin{enumerate}[label=(\alph*)]
			\item $ \kappa \in \DD^p$;
			\item $\kappa$ is $p$-superharmonic on $X\setminus K$ and vanishes on $K$.
			\item $ \kappa(x) \to \infty $ as $ x \to \infty $, i.e.,  for every $ M > 0 $  there exists $F \subseteq X $ finite such that
			$\kappa > M $ in $ X \setminus F $.
		\end{enumerate}
	\end{theorem}

\begin{proof}
	According to Theorem~\ref{prop:LPS}, there exists a function  $ f \in \DD^p(X) $ such that $ f(x) \to \infty $ as $ x \to \infty $. By translating and truncating $ f $, we may arrange that $ f=0 $ on $K$ and $ f \geq 0 $ on $ X \setminus K $.
	
	Consider the non-empty, closed and convex subset 
	\[\mathcal{K}=\set{h\in D^p: h\geq f \text{ on } X\setminus K,\, h=0 \text{ on } K},\]
	of the uniformly convex Banach space $(D^p,\norm{\cdot}_{o,p})$, where $o\in K$. Hence, there exists a minimizer $\kappa$ of $\mathcal{K}$ in $D^{p}$ with respect to $\norm{\cdot}_{o,p}$. Since $o\in K$, $\kappa$ is also a minimizer for $\E_p$ on $\mathcal{K}$. Hence, for every $t>0$, and ${0\leq \phi}\in C_c(X)$ with support outside of $K$, we have $\kappa +t\phi \in \mathcal{K}$, and 
	\[ 0\leq \frac{\E_p(\kappa +t\phi)-\E_p(\kappa)}{t}= \frac{1}{2}\sum_{x,y\in X}b(x,y)\frac{\abs{\nabla_{x,y}\kappa + t\nabla_{x,y}\phi}^p- \abs{\nabla_{x,y}\kappa}^p}{t}.\]
	From here, we can continue as in Proposition~\ref{existence potential}\ref{existence potential i} to conclude that
	\[ \Delta_{p} \kappa \geq 0 \quad \text{ on } X\setminus K.\qedhere\]
\end{proof}

A proof of this statement for the special case $p=2$ (and with a different strategy) can also be found in \cite{Hagen}.

\section{Applications}\label{s:Appl}
We want to give some examples of graphs and classify their type. Moreover, we show some simple consequences of the preceding characterizations.

\begin{remark}\label{rem} Let $(X,b,m)$ be a graph.
	\begin{enumerate}[label=(\alph*)]
		\item If $V\subsetneq X$, then for all $\phi\in C_c(V)$,
		\[ \E_{p}(\phi)= \frac{1}{2}\sum_{x,y\in V}\abs{\nabla_{x,y}\phi}^p+ \sum_{x\in V}\phi^p(x)\sum_{y\in \partial_{e}V}b(x,y).\]
		By the connectedness of $X$, $\mu(x):=\sum_{y\in \partial_{e}V}b(x,y)\geq 0$ for all $x\in X$, and since $V\subsetneq X$, there exists $x_0$ such that $\mu(x_0)>0$. Hence, $\mu$ is a Hardy weight. By Proposition~\ref{prop:Hardy}, the proper subgraph $(V,b,m)$ embedded in $(X,b,m)$ by taking Dirichlet boundary conditions in $X\setminus V$ is $p$-hyperbolic for any $p>1$.
		\item If $X$ is finite then the corresponding graph is $p$-parabolic for any $p\in (1,\infty)$ since 
		$C(X)=C_c(X)=\DD_0^p=\DD^p$ and we can apply Proposition~\ref{prop:D=D_{0}^{p}}~\ref{D2}.
		\item If the graph is $p$-parabolic then it is also $q$-parabolic for any $q>p$. Indeed, by Proposition~\ref{prop:D=D_{0}^{p}}~\ref{D4} there exists a $p$-null sequence $(e_n)$ in $C_c(X)$ with $0\leq e_n\leq 1$ for all $n\in \NN$. Hence, $\abs{\nabla_{x,y}e_n}^p\geq \abs{\nabla_{x,y}e_n}^q$ for all $q>p$, and $(e_n)$ is a $q$-null sequence, i.e., the graph is $q$-parabolic. This motivates the introduction of the parabolic index, defined as the infimum of all $p$ such that the graph is $p$-parabolic. For more information see e.g. \cite{S21,SY93Para,Y77}.
	\end{enumerate}
\end{remark}
We will consider in the following only infinite graphs.

\begin{example}\label{ex:star}
	A \emph{star graph} with center $0$ is a (connected) graph $G=(\NN_0,b,m)$ such that every vertex is only connected to the center $0$, see Fig.~\ref{f:2}~\eqref{f:star}. This is an example of a locally summable graph which is not locally finite.
	
	Note that  for any function $f\in \FF^p$, and $n> 0$,
	\[\Delta_{p}f(n)=\frac{b(n,0)}{m(n)}\p{\nabla_{n,0}f}. \]
	Hence,
	\[\Delta_pf(0)= \frac{1}{m(0)}\sum_{n=1}^\infty b(0,n)\p{\nabla_{0,n}f}=- \sum_{n=1}^\infty \frac{m(n)}{m(0)}\Delta_{p}f(n).\]
	
	We want to see if the graph is $p$-parabolic. By the calculation above it is simple to see that all positive superharmonic functions are harmonic, and hence there cannot exist a Green's function. By Proposition~\ref{thm:criticalNew}, the star graph is $p$-parabolic for $p\in (1,\infty)$ independently of $b$ (as long as the graph is locally summable). 
	Alternatively, it is also not difficult to apply e.g.  Proposition~\ref{Ahlfors} (since there are actually only two cases to consider).
\end{example}
In a hand waving way this observation makes sense, since one can think of hyperbolicity like growing fast at infinity which is just not the case for a star graph for $p<\infty$. 
\begin{example}
	Let us also consider the weighted line graph $\NN_b=(\NN_0,b,m)$ defined via $b(n,m)> 0$ if and only if $\abs{n-m}=1$, $n,m\in \NN_0$, see Fig.~\ref{f:2}~\eqref{f:line}. By using Theorem~\ref{Area} it is immediate that $\NN_b$ is $p$-parabolic if and only if
	 \[\sum_{n=1}^\infty b^{\frac{-1}{p-1}}(n,n+1)=\infty.\]
	 
	 In particular, the standard line graph $(\NN_0,b,m)$ with $m=1$, $b(n,n+1)=1$ for $n\in\NN_0$, is $p$-parabolic. Moreover, the standard line graph with Dirichlet boundary condition at $0$ is $p$-hyperbolic by Remark~\ref{rem}, which was observed already by Hardy.
\end{example}

\begin{example}
	A \emph{wheel graph} with center $0$ is a graph $G=(\NN_0,b,m)$ such that every vertex is connected to the center $0$, and $b(n,n+1)>0$ for all $n\in\NN$. Hence, it is a line graph glued on the spikes of a star graph, see Fig.~\ref{f:2}~\eqref{f:wheel}. This is interesting  since we have just seen that the star graph is always $p$-parabolic for $p<\infty$, whereas the weighted line graph could be $p$-hyperbolic.
	
	We note that the Laplacian has the following explicit form in this example:	
	\begin{align*}
		m(n)\Delta_pf(n)&=b(n,0)\p{\nabla_{n,0}f}+ b(n,n\pm1)\p{\nabla_{n, n\pm 1}f},\qquad n\geq 2,\\
		m(1)\Delta_pf(1)&=b(1,0)\p{\nabla_{1,0}f}+ b(1,2)\p{\nabla_{1, 2}f},\\
		m(0)\Delta_pf(0)&=\sum_{n=1}^{\infty}b(0,n)\p{\nabla_{0,n}f}= -\sum_{n=1}^{\infty}m(n)\Delta_{p}f(n).
	\end{align*}
	Hence, we obtain the same formula as for the star graph and get that it is always $p$-parabolic if $p<\infty$.
\end{example}
If we glue the star and the line graph differently, we also get a different behaviour with respect to parabolicity.
\begin{example}\label{ex:starline}
	A \emph{star-line graph} (or chain chomp) with center $0$ is a graph $G= \NN_s \oplus \NN_l$ consisting of a star graph $\NN_s$ with center $0_s$ and a line graph $\NN_l$ with origin $0_l$, where we identify the center and the origin, i.e. $0=0_s=0_l$, see Fig.~\ref{f:2}~\eqref{f:starline}. By the Khas'miniski\u{\i} test, Proposition~\ref{khasm}, it is obvious that $G$ is $p$-parabolic for $p\in (1,\infty)$ if and only if both  $\NN_s$ and $\NN_l$ are $p$-parabolic, which is the case if also  $\NN_l$ is $p$-parabolic.
\end{example}

\begin{figure*}[t!]
	\centering
	\begin{subfigure}[t]{0.4\textwidth}
		\centering
		
			\begin{tikzpicture}[scale=.7,
				every node/.style={circle, draw, fill=gray!30, minimum size=1mm},
				invisible/.style={draw=none, fill=none}
				]
				
				\node (o) at (0,0) {$0$};
				
				\foreach \angle in {0, 90, 135, 180, 225, 270, 315} {
					\node[anchor=center] at (\angle:2cm) (v\angle) {};
					\draw (o) -- (v\angle);
				}
				
				\foreach \angle in {45} {
					\node[invisible] at (\angle:2cm) (v\angle) {};
					\draw[dashed] (o) -- (v\angle);
				}
				\node[invisible] at (45:2cm) {$\ddots$};
			\end{tikzpicture}
		\caption{Star graph with center $0$}\label{f:star}
	\end{subfigure}%
	~ 
	\begin{subfigure}[t]{0.4\textwidth}
		\centering
		\begin{tikzpicture}[
			scale=.5,
			every node/.style={circle, draw, fill=gray!30, minimum size=4mm},
			invisible/.style={draw=none, fill=none}
			]
			
			\node (v1) at (0,0) {$0$};
			\node (v2) at (3,0) {};
			\node (v3) at (6,0) {};
			\node (v4) at (9,0) {};
			\node[invisible] (dots) at (11,0) {};
			
			\draw (v1) -- (v2);
			\draw (v2) -- (v3);
			\draw (v3) -- (v4);
			\draw[dashed] (v4) -- (dots);
			
		\end{tikzpicture}
		\caption{Line graph}\label{f:line}
	\end{subfigure}
	 \begin{subfigure}[t]{0.4\textwidth}
	\centering
	
\begin{tikzpicture}[scale=.7,
	every node/.style={circle, draw, fill=gray!30, minimum size=1mm},
	invisible/.style={draw=none, fill=none,scale=0.9}
	]
	
	\node (o) at (0,0) {$0$};
	
	\foreach \angle in {0, 90, 135, 180, 225, 270, 315} {
		\node[anchor=center] at (\angle:2cm) (v\angle) {};
		\draw (o) -- (v\angle);
	}
	
	\node[invisible] at (45:2cm) (v45) {$\ddots$};
	
	\draw[dashed] (o) -- (v45);
	\draw[dashed] (v0) -- (v45);
	
	\draw (v135) -- (v90);
	\draw (v180) -- (v135);
	\draw (v180) -- (v225);
	\draw (v225) -- (v270);
	\draw (v270) -- (v315);
	\draw (v315) -- (v0); 
\end{tikzpicture}
	\caption{Wheel graph with center $0$}\label{f:wheel}
\end{subfigure}~\begin{subfigure}[t]{0.4\textwidth}
	\centering
	
	\begin{tikzpicture}[scale=.7,
		every node/.style={circle, draw, fill=gray!30, minimum size=1mm},
		invisible/.style={draw=none, fill=none},
		line/.style={circle, draw, fill=gray!50, minimum size=1mm}
		]
		
		\node (o) at (0,0) {$0$};
		
		\foreach \angle in {0, 90, 135, 180, 225, 270, 315} {
			\node[anchor=center] at (\angle:2cm) (v\angle) {};
			\draw (o) -- (v\angle);
		}
		
		\foreach \angle in {45} {
			\node[invisible] at (\angle:2cm) (v\angle) {};
			\draw[dashed] (o) -- (v\angle);
		}
		\node[invisible] at (45:2cm) {$\ddots$};
		
		\draw (v0) -- (0:3cm);
		\draw (0:3cm) -- (0:4cm);
		\draw[dashed] (0:4cm) -- (0:5cm);
		
		\node[line, anchor=center] at (0:1cm) (0) {};
		\node[line, anchor=center] at (0:2cm) (0) {};
		\node[line, anchor=center] at (0:3cm) (0) {};
		\node[line, anchor=center] at (0:4cm) (0) {};
	\end{tikzpicture}
	\caption{Star-line graph with center $0$}\label{f:starline}
\end{subfigure}

	\caption{Graphs of Examples~\ref{ex:star} - \ref{ex:starline}}
	\label{f:2}
\end{figure*}

The example also motivates the following result. It is well known, see e.g. \cite{GT99,T99}, but using the  Khas'miniski\u{\i}-type characterization of $p$-parabolicity provides a particularly short proof.

{An \emph{end} of a graph with respect to a finite set $K\sse X$ is a connected component of the complement $X\setminus K$.
\begin{proposition}
	If all ends of a graph with respect to a fixed finite set with finite exterior boundary are $p$-parabolic then the graph itself is $p$-parabolic.
\end{proposition}

\begin{proof}
By assumption, we know that there are at most finitely many ends. Then using Propositions~\ref{khasm} and \ref{prop:weak_char} yields the result.
\end{proof}}
In turn, the following can be said about getting $p$-hyperbolicity.
\begin{proposition}
	Let  $G=(X,b,m)$ be a graph. Let $W\sse X$ such that $\partial_{i}W$ is finite. If the subgraph $G_W:=(W,b|_{W\times W}, m|_W)$ is $p$-hyperbolic, then so is the whole supergraph $G$.
\end{proposition}

\begin{proof}
	Since $G_W$ is $p$-hyperbolic, the weak maximum principle does not hold (Proposition~\ref{prop:weak_char}), i.e., there exists a non-constant, bounded above function $f\in \FF^p$ and a number $\gamma$ such that \[\sup_{\Omega_\gamma}\Delta_p^Wf\leq 0.\]
	Here, the Laplacian $\Delta_p^W$ is understood with respect to the weight $b|_{W\times W}$.
	By the strong maximum principle, Lemma~\ref{max_principle}, applied to the graph $G_W$, $f$ does not attain its supremum on the set $\Omega_\gamma$ and therefore it cannot attain its supremum on the finite set $\partial_iW$. Hence there exists $\gamma \leq \beta< f^*=\sup f$ such that $\Omega_\beta \cap \partial_i W=\emptyset$ and therefore $\partial_e\Omega_\beta\subseteq W$.

If $u\in C(X)$ is defined by
	\[u(x):=\begin{cases}
		f(x),
\qquad &x\in W;\\
		\gamma, &x\in X\setminus W.
	\end{cases}\]
 then, for every $x\in \Omega_\beta$
	\begin{multline*}
		m(x)\Delta_pu(x)=\sum_{y\in X}b(x,y)\p{\nabla_{x,y}u} \\ =\sum_{y\in W}b(x,y)\p{\nabla_{x,y}f}=m(x)\Delta_p^W f(x)\leq 0. \qedhere
	\end{multline*}
\end{proof}

We now recall some well-known examples of locally finite graphs for convenience.

\begin{example}
	The Euclidean lattice $\ZZ^d$ is $p$-parabolic if and only if $d\leq p$, see e.g. \cite{M77}. In this article, a flow is constructed to apply a version of Theorem~\ref{KNR} in order to get $p$-hyperbolicity for $d>p$. In turn, $p$-parabolicity is shown via a result in the spirit of Theorem~\ref{Area}.
\end{example}

Next we show another consequence of Theorem~\ref{Area}, and apply it thereafter for some special cases. For the linear case, see the proof of Theorem~3.17 in \cite{FR25}.

\begin{proposition}\label{prop:GreenExplizit}
	Let a $p$-hyperbolic locally finite model graph with respect to a given $x_0\in X$. If the curvature ratio function
	\[ \kappa(r):= \frac{k_+(r)}{k_-(r)}> 0, \qquad r\in \NN_0,\]
	is constant and strictly larger than $1$ outside of $B_{r_0}$ for some $r_0\in \NN_0$, then the Green's function $g$ with respect to $x_0\in X$ is radial and  given by
	\[g(r)= \frac{1}{(\kappa^{\frac{1}{p-1}}-1)\partial B(r-1)^{\frac{1}{p-1}}}, \qquad r>r_0. \qedhere\]
\end{proposition}
\begin{proof}
	It follows from  \eqref{eq model} that
	\[ \kappa\, \partial B(r-1)= \partial B(r), \qquad r> r_0.\]
	Hence, by Theorem~\ref{Area}~\ref{Area2},
	\[ g(r)= \sum_{k=r}^\infty\left(\kappa\, \partial B(k-1) \right)^{-\frac{1}{p-1}} = \left(\kappa\, \partial B(r-1) \right)^{-\frac{1}{p-1}}+ \kappa^{- \frac{1}{p-1}} \sum_{k=r+1}^\infty\partial B(k-1)^{-\frac{1}{p-1}}.  \]
	Rearranging yields,
	\[g(r)= \frac{1}{(\kappa^{\frac{1}{p-1}}-1)\partial B(r-1)^{\frac{1}{p-1}}}, \qquad r> r_0.\qedhere\]
\end{proof}

Let us turn to two prominent model graphs: homogeneous regular trees and anti-trees.
\begin{example}[$\TT_{d+1}$, $d\in\NN$]\label{ex:T}
	Let $d\in \NN$. A graph with root $x_0\in X$ is a \emph{rooted homogeneous $(d+1)$-regular tree}, denoted by $\TT_{d+1}$, if $b(X\times X)=\set{0,1}$, $m=1$, $k_+(x)= d$, $k_-(x)=1$ for all $x\in \TT_{d+1}\setminus \{ x_0\}$, $b(S_r\times S_r)= \set{0}$ for all $r\geq 0$, and the root $x_0$ has $d$ children and no parent, i.e., $k_-(x_0)=0$ and $k_+(x_0)=d$, see Fig.~\ref{f:3}~\eqref{f:T}.
	
	Since $\partial B(r)= d^{r+1}$  for all $r>0$, we can apply Theorem~\ref{Area}, and get that  the graph $\TT_{d+1}$ is $p$-hyperbolic for all $p>1$ and $d> 1$. Since $\kappa(r)= d$, $r>0$, we can use Proposition~\ref{prop:GreenExplizit}, and get that the Green's function is given explicitly by
	\[g(r)= \frac{1}{(d^{\frac{1}{p-1}}-1)d^{\frac{r}{p-1}}}, \qquad r> 0.\]
\end{example}
An approach of specifying the type via calculating the capacity of the unrooted homogeneous regular tree can be found in \cite[Section~3.5]{Prado}. This is almost $\TT_{d+1}$, but now also the root has $d+1$ neighbors. A similar strategy also works for $\TT_{d+1}$.

\begin{example}[Anti-trees]\label{ex:AT}
		Let  $G=(X,b,m)$ be a model graph with respect to $x_0\in X$ with standard weights $b(X\times X)\sse \set{0,1}$, $m=1$. Moreover, for all $r\in\NN_0$, let $s(r)$ denote the number of vertices in the sphere $S_r$. Then, $G$ is called \emph{anti-tree with sphere size $s$} if $k_{+}(r)=s(r+1)$ and $k_-(r+1)=s(r)$ for all $r\in \NN_0$, see Fig.~\ref{f:3}~\eqref{f:AT}. Hence, for all $r\geq 1$ we have
	\[\kappa (r)= \frac{s(r+1)}{s(r-1)}, \quad  \partial B(r)=s(r)s(r+1). \]
	Note that there are no restrictions for the neighbors within a sphere. Moreover, the standard line graph on $\NN_0$ is an anti-tree with sphere size $s=1$.
	
	Let us consider the anti-tree with $s(r)=d^r$ for some $d\in \NN \setminus \set{1}$ and $r\geq 0$. By using Theorem~\ref{Area}, we see that this anti-tree is $p$-hyperbolic for all $p>1$. Moreover, since $\kappa(r)= d^2$, $r>0$, we can use Proposition~\ref{prop:GreenExplizit}, and get that the corresponding Green's function is given explicitly by
	\[g(r)= \frac{1}{(d^{\frac{2}{p-1}}-1)d^{\frac{2r+1}{p-1}}}, \qquad r> 0.\]
	
Let us consider the anti-tree with $s(r)=r+1$ for all $r\geq 0$. By using Theorem~\ref{Area}, we see that this anti-tree is $p$-hyperbolic if and only if $p<3$.
\end{example}

	\begin{figure}[htbp]
	\centering
	\begin{subfigure}{0.45\textwidth}
		\centering
		\begin{tikzpicture}[
			level distance=1.5cm,
			level 1/.style={sibling distance=2.5cm},
			level 2/.style={sibling distance=1cm},
			level 3/.style={sibling distance=.5cm},
			every node/.style={circle, draw, fill=gray!30, minimum size=1mm},
			invisible/.style={draw=none, fill=none}
			]
			
			\node (o1) {$x_0$}
			child { node {}
				child { node {}
					child[dashed] { node[invisible] {} }
					child[dashed] { node[invisible] {} }
				}
				child { node {}
					child[dashed] { node[invisible] {} }
					child[dashed] { node[invisible] {} }
				}
			}
			child { node {}
				child { node {}
					child[dashed] { node[invisible] {} }
					child[dashed] { node[invisible] {} }
				}
				child { node {}
					child[dashed] { node[invisible] {} }
					child[dashed] { node[invisible] {} }
				}
			};
			
		\end{tikzpicture}
		\caption{$\TT_{2+1}$}\label{f:T}
	\end{subfigure}
	\begin{subfigure}{0.45\textwidth}
		\centering
		\begin{tikzpicture}[scale=0.7,
			every node/.style={circle, draw, fill=gray!30, minimum size=1mm},
			level1/.style={circle, draw, fill=gray!20},
			level2/.style={circle, draw, fill=gray!10},
			invisible/.style={draw=none, fill=none},
			]
			
			\node (o) at (0,0) {$x_0$};
			
			\node[level1] (a1) at (-1.5,-2) {};
			\node[level1] (a2) at (1.5,-2) {};
			
			\foreach \x in {a1,a2}
			\draw (o) -- (\x);
			
			\node[level2] (b1) at (-2.5,-4) {};
			\node[level2] (b2) at (0,-4) {};
			\node[level2] (b3) at (2.5,-4) {};
			
			\foreach \x in {a1,a2}
			\foreach \y in {b1,b2,b3}
			\draw (\x) -- (\y);
			
			\node[invisible] (c1) at (-3.5,-6) {};
			\node[invisible] (c2) at (-1,-6) {};
			\node[invisible] (c3) at (1,-6) {};
			\node[invisible] (c4) at (3.5,-6) {};
			
			\foreach \x in {b1,b2,b3}
			\foreach \y in {c1,c2,c3,c4}
			\draw[dashed] (\x) -- (\y);
			
			\draw (b1) -- (b2);
			
		\end{tikzpicture}
		\caption{An anti-tree with $s(r)=r+1$}\label{f:AT}
	\end{subfigure}
\caption{Graphs of Examples~\ref{ex:T} - \ref{ex:AT}}\label{f:3}
\end{figure}

\appendix

\section{The Weak Comparison Principle}
In this paper, we sometimes use a comparison principle which is stated next. It also goes under the names maximum or minimum principle. Several variants are available in the literature, see, e.g., \cite{Prado} and \cite{KimChung2010}. We report below a version well suited to our purposes which appears in \cite{F:AAP} in the setting of $p$-Schrödinger operators. Since the proof is quite short, we reproduce it here for the convenience of the reader. Note that the condition in the statement that $(v-u)\wedge 0$ attains a minimum in $V$ is automatically satisfied when $V$ is finite.

\begin{lemma}
\label{lem:WCPNonNegativePotential}
	Let $V\subsetneq X$. Furthermore, let $u, v\in \FF(V)$ such that
	\begin{align*}
	\begin{cases}
	\Delta_pu&\leq \Delta_pv \quad \text{on } \phantom{\partial}V, \\	
	\phantom{H}u&\leq \phantom{H}v \quad \text{on }\, \partial_e V.
	\end{cases}
	\end{align*}
Assume that $(v-u)\wedge 0$ attains a minimum in $V$.
	Then, $u\leq v$ on $V$.
	
	Moreover, in each connected component of $V$ we have either $u=v$, or $u<v$.
\end{lemma}
\begin{proof} Without loss of generality, we can assume that $V$ is connected. Otherwise, we do the following proof in every connected component of $V$.

Assume that there exists $x\in V$ such that $v(x)\leq u(x)$. Since  $(v-u)\wedge 0$ attains a minimum in $V$, there exists $x_0\in V$ such that $v(x_0)-u(x_0)\leq 0$, and $v(x_0)-u(x_0)\leq v(y)-u(y)$ for all $y\in V$. Since $u\leq v$ on $\partial_e V$, we get $\nabla_{x_0,y} v\leq \nabla_{x_0,y}u$ for all $y\in V\cup \partial_e V$. Furthermore, we have
\begin{multline*}
	0\leq m(x_0)\bigl(\Delta_pv(x_0)-\Delta_pu(x_0) \bigr) 
	= \sum_{y\in V\cup \partial_e V}b(x_0,y)\left( \p{\nabla_{x_0,y}v} -\p{\nabla_{x_0,y}u} \right)\\
	\leq 0,
\end{multline*}
where the second inequality follows from the monotonicity of $\p{\cdot}$ on $\RR$. Thus, we have in fact equality above.

Hence, we get that $u(y)-v(y)$ is a non-negative constant for all $y\sim x$. By iterating this argument and using that $V$ is connected, we get that $u(y)-v(y)$ is a non-negative constant for all $y\in V\cup \partial_e V$. Since $u\leq v$ on $\partial_e V$, we conclude that $u=v$ on $V$.
\end{proof}

The proof above also shows the following.

\begin{lemma}[Strong Maximum Principle]\label{max_principle}
	Let $V\subsetneq X$ be connected and let $u$ be a $p$-subharmonic function in $V$ attaining its maximum in $ V$. Then $u$ is constant in $V$.
\end{lemma}

\section{Alternative Proof of Theorem~\ref{khasm}}\label{sec:Alternative}
This section is devoted to describing an alternative proof of the Khas'miniski\u{\i}-type characterization in the quasi-linear setting on graphs. The main idea is to use the existence of solutions to the so-called obstacle problem as in \cite{Valtorta, ValtortaDiss}. Even though the obstacle problem is well studied in the quasi-linear setting in the continuum and on metric length spaces, see e.g. \cite{Bjoern, HKM}, we are not aware of an analysis on discrete locally summable graphs apart from the locally finite setting in \cite{HKO}. Here we use a result of Subsection~\ref{sec:LPS}.

Recall the notation
\[
\E_{p,{V}}(u,v):=\frac{1}{2} \sum_{x,y \in {V}}b(x,y)\p{\nabla_{x,y}u}\nabla_{x,y}v,
\]
if $V\subseteq X$ and $u,v\in C(V)$. Moreover, on the diagonal, we set \[\E_{p,V}(u):=\E_{p,V}(u,u)=\frac{1}{2}\norm{\nabla u}^p_{p,b,V}.\]

Next, we define the obstacle problem and study the existence and uniqueness of solutions to this problem. The counterpart in the continuum can be found in \cite{HKM}.

Given functions $\psi,\theta\in C(X)$ and a subset $V\subset X$ we define the set $K_{\psi, \theta}(V)$ by
\[K_{\psi, \theta}(V):= \set{ v\in \DD^p(\overline{V}): v\geq \psi \text{ in } V, v=\theta \text{ on } \partial_e V}. \]

Recall that, having fixed a point $o\in \overline{V}$ then $\DD^p(\overline{V})$ is a uniformly convex Banach space with respect to the norm
\[
\norm{u}_{o,p}=\left(
|u(o)|^p+\E_{p,\overline{V}}(u)
\right)^{1/p}=
\left(
|u(o)|^p+\frac{1}{2}\norm{\nabla u}^p_{p,b,\overline{V}}
\right)^{1/p},
\]
and that changing the point $o$ gives rise to an equivalent norm.

\begin{lemma}\label{K_properties} Let $V\sse X$. The set $K_{\psi, \theta}(V)$ is a convex closed subset of $\DD^p(\overline{V})$.
\end{lemma}

\begin{proof}
	The convexity of $K_{\psi, \theta}:=K_{\psi, \theta}(V)$ is clear. To prove that it is closed, assume that $(u_i)$ is a sequence in $K_{\psi, \theta}$ such that $\norm{u_i-u}_{o,p}\to 0$. Arguing as in the proof of Lemma~\ref{lem:uniform} shows that $u_i\to u$ pointwise in $\overline{V}$ and it  follows that $u\in \DD^p(\overline{V})$ satisfies $u\geq \psi$ in $V$ and $u=\theta$ on $\partial_eV$, so that $u\in K_{\psi, \theta}$ which therefore is closed.
\end{proof}

\begin{definition}
	We say that $u\in K_{\psi, \theta}(V)$ is a \emph{solution to the obstacle problem} on $V$ with obstacle $\psi$ and boundary data $\theta$ if for all $v\in K_{\psi, \theta}(V)$ we have
	\[ \E_{p,\overline{V}}(u, v-u) \geq 0. \]
\end{definition}

Note that for $u,v \in \DD^p(\overline {V})$ the sum on the left-hand side converges absolutely and, since $u-v$ vanishes on $\partial_eV$, if $V$ is finite
then $\E_{p, \overline V}(u,v-u)$ is defined provided $u\in \FF^p(V)$ and
\[
\E_{p, \overline V}(u,v-u) = \langle \Delta_p u, v-u\rangle
\]
(compare with Lemma~\ref{lem:GreensFormula}). Moreover, if $u$ is a solution to the obstacle problem in $K_{\psi, \theta}(V)$ and $0\leq \phi\in C_c(V)$, then $v=u+\phi \in K_{\psi, \theta}(V)$ and therefore
$0\leq \E_{p,\overline{V}}(u, \mathds{1}_x)= \Delta_pu(x)$ for all $x\in V$, i.e., $u$ is superharmonic in $V$.

It is classical that since the obstacle problem originates from a variational integral, solutions  are precisely the  minimizers of $\E_p(u)$ in $K_{\psi,\theta}$. For completeness we provide a proof of this fact in the following lemma.

\begin{lemma}\label{lem:minimziers} A function  $u\in K_{\psi,\theta}(V)$ is a solution to the obstacle problem if and only if
	\[
	\E_{p,\overline V} (u)= \min\{\E_{p,\overline V} (v)  \,:\, v\in K_{\psi,\theta}(V)\}.
	\]
\end{lemma}
\begin{proof}
	One implication follows easily from H{\"older}'s inequality: let $u\in K_{\psi,\theta}(V)$ be a solution to the obstacle problem. By definition, for every $v\in K_{\psi,\theta}(V)$
	\begin{equation*}
	\begin{split}
		0&\leq \E_{p,\overline V} (u,v-u)
		= \E_{p,\overline V} (u,v)-\E_{p,\overline V} (u)\\
		&=  \frac{1}{2}\sum_{x,y\in \overline V}  b(x,y) \p{\nabla_{x,y} u}\nabla_{x,y}v   - \E_{p,\overline V} (u)\\
		&\leq \frac{1}{2}\Biggl(\sum_{x,y\in \overline V}b(x,y) |\nabla_{x,y} u|^p\Biggr)^{1-1/p} \Biggl(\sum_{x,y\in \overline V}b(x,y)|\nabla_{x,y} v|^{p} \Biggr)^{1/p} -\E_{p,\overline V} (u)\\
		&=\E_{p,\overline V} (u)^{1-1/p} \E_{p,\overline V} (v)^{1/p} -\E_{p,\overline V} (u),
		\end{split}
	\end{equation*}	
	and the required conclusion follows.
	
	For the reverse implication we use the arguments of the proof of Proposition~\ref{existence potential}: suppose that
	\[
	\E_{p,\overline V} (u)= \min\{\E_{p,\overline V} (v)  \,:\, v\in K_{\psi,\theta}(V)\}.
	\]
	Fix $v\in K_{\psi,\theta}(V)$ and let $\varphi= v-u$. By convexity, for every $t\in (0,1]$, we have
	$u+t\varphi=(1-t)u +tv \in K_{\psi,\theta}(V)$ and therefore
	\[
	0\leq \frac{\E_{p,\overline V} (u+t\varphi,u+t\varphi)-\E_{p,\overline V} (u,u)}t =\frac{1}{2}
	\sum_{x,y\in \overline V} b(x,y) \frac{|\nabla_{x,y} u + t \nabla_{x,y} \varphi|^p-|\nabla_{x,y} u|^p}t.
	\]
	As $t\to 0$, the fraction inside the summation tends to
	\[
	p|\nabla_{x,y} u|^{p-2}\nabla_{x,y} u \nabla_{x,y} \phi ,
	\]
	and an application of the mean value theorem shows that
	\[
	\left|\frac{|\nabla_{x,y} u + t \nabla_{x,y} \varphi|^p-|\nabla_{x,y} u|^p}t\right|
	\leq
	C_p( |\nabla_{x,y} u |^p +|\nabla_{x,y} \varphi|^p)\in \ell^1(E_{\overline{V}}, b),
	\]
	so that, letting $t\to 0$ and applying dominated convergence yield
	\[
	0\leq\frac{1}{2} \sum_{x,y\in \overline V}b(x,y) |\nabla_{x,y} u|^{p-2}\nabla_{x,y} u \nabla_{x,y} \phi = \E_{p,\overline V} (u,v-u),
	\]
	as required to prove that  $u$ is a solution to the obstacle problem.
\end{proof}

In the remaining part of this subsection, we will discuss uniqueness and existence of solutions to the obstacle problem.

\begin{theorem}\label{thm:sol-obstacle}
	Let $V\sse X$. Suppose that $K_{\psi, \theta}(V)$ is not empty. Then the obstacle problem admits a unique solution $u\in K_{\psi, \theta}(V)$.
\end{theorem}

\begin{proof}
	We may assume that the point $o$ in the definition of the norm $\norm{\cdot}_{o,p}$ belongs to $\partial_eV$. Since $K_{\psi, \theta}(V)$ is a closed convex set in a uniformly convex Banach space it admits a unique element $u$ of minimum norm, namely, for all $v\in K_{\psi, \theta}(V)$,
	\[
	|u(o)|^p+\E_{p,\overline V}(u)\leq |v(o)|^p+\E_{p,\overline V}(v).
	\]
	Since $u(o)=v(o)=\theta (o)$, it follows that for all such $v$
	\[
	\E_{p,\overline V}(u)\leq \E_{p,\overline V}(v),
	\]
	and the conclusion follows from Lemma~\ref{lem:minimziers} above.
\end{proof}

We will also need the following two lemmas. Their counterparts in the continuum are given in \cite[Lemmas~3.11 and~3.22]{HKM}. The first is an application of the Green's formula and the second says that the solution to the obstacle problem in $K_{\psi, \theta}(V)$ is the smallest superharmonic function in $K_{\psi, \theta}(V)$.

\begin{lemma}\label{lem:3-11}
	Let $V\sse X$. If either
	\begin{itemize}
		\item $u\in \DD^p(V)$ is (super-)harmonic in $V$ and $v \in \DD_{0}^{p}(V)$ (respectively, $0\leq v\in \DD_{0}^{p}(V)$) satisfies $v =0$ on $\partial_e V$,
	\end{itemize}	
	or
	\begin{itemize}
		\item $u\in \FF^p(V)$ is (super-)harmonic in $V$ and $v \in C_c(V)$ (respectively, $0\leq v\in C_c(V)$),
	\end{itemize}	
	we have
	\[ \E_{p,\overline{V}}(u,v) \overset{( \geq )}{=} 0.  \]
\end{lemma}
\begin{proof}
	By Lemma~\ref{lem:GreensFormula} and Lemma~\ref{lem:GreensFormulaD}, we have in both cases $  \E_{p,\overline{V}}(u,v)= \ip{\Delta_p u}{v}_V.$  Using the additional assumptions on $u$ and $v$ it is easy to show that the stated conclusion holds.
\end{proof}

\begin{lemma}\label{lem:3-22}
	Let $W\sse X$ be finite. 
	Let $u$ be a solution to the obstacle problem in $K_{\psi, \theta}(W)$. If $v\in \FF^p(\overline{W})$, respectively $v\in \DD^p(\overline{W})$, is superharmonic in $W$ and  $u\wedge v\in K_{\psi, \theta}(W)$, then $v\geq u$ in $W$.
\end{lemma}
\begin{proof}
	Since $u, u\wedge v\in K_{\psi, \theta}(W)$ and $W$ is finite, we have that $u- u\wedge v \in D_{0}^{p}(\overline{W})=C_c(W)=C(W)$. Therefore all of the following sums converge absolutely. Moreover, by Lemma~\ref{lem:3-11},  $\E_{p,\overline{W}}(v,u-u\wedge v)\geq 0$. Since $u$ is a solution to the obstacle problem we also  have $\E_{p,\overline{W}}(u,u-u\wedge v)\leq 0$. Altogether,
	\begin{multline}
		\label{ineq:uv}
		0 \leq \E_{p,\overline{W}}(v,u-u\wedge v) - \E_{p,\overline{W}}(u,u-u\wedge v)\\
		= \frac{1}{2}\sum_{x,y\in \overline{W}} b(x,y)(\p{\nabla_{x,y}v}- \p{\nabla_{x,y}u})(\nabla_{x,y}u - \nabla_{x,y}(u\wedge v)).
	\end{multline}
	Moreover,  a case by case argument which uses the  monotonicity of $t\to {t}^{\langle p-1\rangle}$  shows that, for every $x,y\in \overline{W}$,
	\begin{multline*}
		(\p{\nabla_{x,y}v}- \p{\nabla_{x,y}u})(\nabla_{x,y}u - \nabla_{x,y}(u\wedge v))
		\\ \leq (\p{\nabla_{x,y}(u\wedge v)}- \p{\nabla_{x,y}u})(\nabla_{x,y}u - \nabla_{x,y}(u\wedge v))\leq 0,
	\end{multline*}
	and, inserting into \eqref{ineq:uv}, we conclude that, for all $x,y \in \overline{W}$ with $x\sim y$, 	
	\[
	(\p{\nabla_{x,y}(u\wedge v)}- \p{\nabla_{x,y}u})(\nabla_{x,y}u - \nabla_{x,y}(u\wedge v))=0,
	\]
	and therefore $\p{\nabla_{x,y}(u\wedge v)}= \p{\nabla_{x,y}u}$ . Since $u=v=\theta$ on $\partial_e V$, it follows that $u= u\wedge v$ on $W$.	
\end{proof}
Now we are in a position {to} give an alternative proof of the characterization of $p$-parabolicity in terms of the existence of Khas'minski\u{\i} potentials, Theorem~\ref{khasm}.

	\begin{proof}[Alternative proof of Theorem~\ref{khasm}]
		The proof is modelled on   \cite[Theorem~2.5]{Valtorta}. According to Theorem~\ref{prop:LPS}, there exists a function  $ f \in \DD^p(X) $ such that $ f(x) \to \infty $ as $ x \to \infty $. By translating and truncating $ f $, we may arrange that $ f=0 $ on $K$ and $ f \geq 0 $ on $ X \setminus K $. Let $ (X_n) $ be an exhaustion of $X$ by finite sets such that $ X_0=K$ and $ f \geq n $ on $X\setminus X_n$.
		
		We claim that there exists an increasing sequence of functions $s_n \in \DD^p(X)$ such that
		\begin{enumerate}
			\item $s_n$ is $p$-superharmonic in $ X\setminus K$ and vanishes on $K$;
			\item $s_n \leq n$ and there exists a finite set $S_n $ such that $ s_n =n $ on $X\setminus S_n$.
		\end{enumerate}
		Moreover, $s_n$ is pointwise bounded and
		\begin{equation*}
			\kappa(x) := \lim_n s_n(x)
		\end{equation*}
		is the required function.
		
		Set $s_0=0$ and assume by induction 
		that we have constructed functions $ s_0 \leq s_1 \leq \ldots \leq s_n $ with the required properties. For ease of notation we set $ s_n = s $, $ S_n = S$ and we are going to construct $s_+=s_{n+1} $.
		
		For every $j\in \NN$ let $ f_j(x) = (j^{-1} f(x))\wedge 1$ so that $ f_j \leq 1 $ and $ f_j =1 $ off the finite set $ \{ x\in X \, : \, f(x) < j \} $ and $ f_j(x) \to 0 $ as $j\to \infty$ pointwise, i.e., locally uniformly.
		
		Let   $ h_j \in K_{\psi_j,\psi_j}(\Omega_j)$ be the solution to the obstacle problem on $ \Omega_j = X_{j+1} \setminus X_0 $ with obstacle and boundary data $ \psi_j =s+f_j $. Hence,
		\begin{itemize}
			\item $ h_j $ is $p$-superharmonic on $ \Omega_j $,
			\item $ h_j = \psi_j $ on $ X_0=K $, and
			\item $ h_j = \psi_j $ on $ \partial_e X_{j+1} $.
		\end{itemize}
		Since $ f_j =1 $ on $X\setminus X_j \supseteq \partial_e X_{j+1} $, $ h_j=s+1 $ on $ \partial_e X_{j+1} $ and for $ j $ large enough $ s=n $ on $X\setminus X_{j+1} $. Thus, for such $ j$'s we have
		\begin{center}
			$ h_j=n+1 $ on $ \partial_e X_{j+1} $,
		\end{center}
		and we can extend it to $ X $ by setting it equal to $ n+1 $ on $X\setminus X_{j+1}$, namely,
		\begin{equation*}
			\tilde{h}_j = \begin{cases}
				h_j & \text{ on } X_{j+1} \setminus X_0 \\
				0 & \text{ on } X_0 \\
				n+1 & \text{ on } X\setminus X_{j+1}.
			\end{cases}
		\end{equation*}
		Note that $\tilde{h}_j$ is a solution to the obstacle problem. For notational convenience we set $ \tilde{h}_j=h_j$.
		
		\textit{1st claim:} $h_j \leq n+1$ on $X_{j+1}$ and therefore on $X$.
		
		Indeed, since $X_{j+1}$ is finite, $h_j \in \DD^p(\overline{X}_{j+1})$ by assumption and we can apply Lemma~\ref{lem:3-22} with $ u=h_j $ and $ v=n+1 $ (so that certainly $ v \in \DD^p(\overline{X}_{j+1}) $ and $u \wedge v = h_j \wedge (n+1) \in K_{\psi_j}$) to conclude.\\

		\textit{2nd claim:} $h_j$ is $p$-superharmonic on $X \setminus X_0 $.
		
		Since by construction $ h_j $ is $p$-superharmonic on $ X_{j+1} $ and it is constant (equal to $n+1$) on $ X\setminus X_{j+1}$, it suffices to check superharmonicity at $ x \in \partial_e X_{j+1} $. Since $ h_j \leq n+1 $ and $ h_j=n+1 $ on $ X\setminus X_{j+1}$, we have, for $ x \in \partial_e X_{j+1} $,
		\begin{equation*}
			m(x)\Delta_p h_j(x) = \sum_{y\in X} b(x,y) \p{\nabla_{x,y}h_j}=\sum_{y\in X} b(x,y) \p{n+1 -h_j(y)} \geq 0.
		\end{equation*}
		Thus $h_j$ is $p$-superharmonic on $X \setminus X_0 $, $ h_j =0 $ on $ K $ and $ h_j=n+1 $ on $ X\setminus X_{j+1}$, provided that $ j$ is sufficiently large.
		
		\textit{3rd claim:} $h_j $ is monotone decreasing as $ j \to \infty $.
		
		To see this, note that the obstacle $ \psi_j=s+f_j=s+ (j^{-1}f)\wedge 1$ is decreasing. Thus applying again Lemma~\ref{lem:3-22} with $u=h_{j+1} $, which solves the obstacle problem on $ X_{j+2} \setminus X_0 $ with obstacle $ \psi_{j+1} \leq \psi_{j} $, and $ v = h_j $, which is $p$-superharmonic on $ X \setminus X_0 \supseteq X_{j+2} $ and $ v \geq \psi_j > \psi_{j+1} $ on $ X_{j+1} \setminus X_0 $, we deduce that
		\begin{equation*}
			h_j \geq h_{j+1}
		\end{equation*}
		as claimed, and thus \[ h_j -s := \rho_j \searrow \rho \geq 0 .\]
		
		\textit{4th claim:} $\rho=0$, and therefore $h_j \searrow s $ pointwise  as $ j \to \infty $.
		
		Recall that $2\E_{p,V}:=\norm{\nabla\cdot}_{p,b,V}^p\leq \norm{\nabla\cdot}_{p,b,X}^p=:\norm{\nabla\cdot}_{p,b}^p $ for $V\sse X$. Since $ \psi_j = s + f_j \in K_{\psi_j,\psi_j}(\Omega_j) $, it follows from the minimizing property of Lemma~\ref{lem:minimziers} that
		\begin{align}\label{estimate_Khas}
			2\E_{p,X_{j+1}}(h_j)
			\leq 2\E_{p,\overline{X}_{j+1}}(\psi_j)
			\leq  \| \nabla s \|^p_{p,b,\overline{X}_{j+1}} + \| \nabla f_j \|^p_{p,b,\overline{X}_{j+1}}
			\leq \| \nabla s \|^p_{p,b} + \| \nabla f_j\|^p_{p,b},
		\end{align}
		whence
		\begin{equation*}
			\| \nabla \rho_{j} \|_{p,b} \leq \| \nabla h_{j} \|_{p,b} + \| \nabla s \|_{p,b} \leq 2 \| \nabla s \|_{p,b} + \| \nabla f_j \|_{p,b}\leq  C,
		\end{equation*}
		where the last inequality follows since $s$ and $f_j$ are constant outside a finite set.
		
		Since
		\begin{equation*}
			|\nabla_{x,y} \rho_j |^{p} \to |\nabla_{x,y} \rho |^{p}, \qquad x,y\in X,
		\end{equation*}
		by Fatou's Lemma
		\begin{equation*}
			\| \nabla \rho \|_{p,b} \leq \liminf_j \| \nabla \rho_j \|_{p,b} \leq C.
		\end{equation*}
		To show that $\| \nabla \rho_j \|_{p,b} \to 0$, so that $ \| \nabla \rho \|_{p,b}=0 $ and therefore, since $ \rho = 0 $ on $ K$,  $ \rho =0$ on $ X $, we consider the function
		\begin{equation*}
			g=\left(s+\frac 12 \rho_{j}\right) \wedge n ,
		\end{equation*}
		which satisfies $ g=0 $ on $X_0 =K $, $ g=n$ on $ X\setminus S$, $g \geq s $ on $ S \setminus X_0 $.
		
		Since $ s $ is the solution to the obstacle problem with obstacle equal to itself, again by Lemma~\ref{lem:minimziers}, we have
		\begin{equation*}
			\| \nabla s+ \frac{1}{2} \nabla \rho_j \|_{p,b}^{p} \geq 2\E_{p,\overline{S \setminus X_0}}(g,g)
			\geq 2\E_{p,\overline{S \setminus X_0}}(s,s) = \| \nabla s \|_{p,b}^p,
		\end{equation*}
		whence using Lemma~2 in \cite{Valtorta} and estimate \eqref{estimate_Khas} above, we deduce that
		\begin{equation*}
			\begin{split}
				\| \nabla s \|_{p,b} \left[ 1+ \delta\left( \frac{\| \nabla \rho_j \|_{p,b}}{\| \nabla s \|_{p,b} + \| \nabla \rho_j \|_{p,b}} \right) \right]^{-1} &\leq \| \nabla (s + \rho_j)\|_{p,b} \\
				&=2 \E^{1/p}_{p,\overline{S \setminus X_{0}}}(h_j,h_j) \leq||\nabla s||_{p,b}+ \| \nabla f_j \|_{p,b},
			\end{split}
		\end{equation*}
		where $\delta=\delta_{\ell^p}$ is the modulus of convexity of the uniformly convex space  $\ell^p(E_X,b)$. Since $ \| \nabla f_j \|_{p,b} \to 0 $ as $ j \to \infty $ and $\delta$ is strictly positive on $(0,2)$ we conclude that
		\begin{equation*}
			\| \nabla \rho_j \|_{p,b} \to 0.
		\end{equation*}
		We have thus proved that $ \rho_j \to 0 $ on $ X $, i.e., $ h_j \searrow s $ on $ X $ which shows the claim.
		
		We are ready to complete the induction process. By the claims above, there exists $\overline{j} $ such that $ \sup_{x\in X_{n+1}} | h_{\overline{j}}(x) -s(x) | < 2^{-n-1} $ and $ \| \nabla \rho_{\overline{j}} \|_{p,b} < 2^{-n} $, so that
		\begin{equation*}
			\| \nabla s_{+} - \nabla s \|_{p,b} < 2^{-n}.
		\end{equation*}
		Now set $ \kappa:=\lim_n s_n $. Then:
		\begin{itemize}
			\item since $ s_{n+1}-s_n < 2^{-n-1} $ it follows that
			\begin{equation*}
				\kappa=\sum_{n=1}^{\infty} (s_n - s_{n-1})
			\end{equation*}
			is finite everywhere;
			\item since $ s_n = n $ on $X\setminus S_n$ and $X\setminus S_n $ is finite, it follows that
			\begin{equation*}
				\kappa \to \infty
			\end{equation*}
			at $ \infty $;
			\item since
			\begin{equation*}
				\| \nabla s_{n+1} - \nabla s_n \|_{p,b} \leq 2^{-n},
			\end{equation*}
			it follows that $ \kappa \in \DD^p(X) $.
			\item Finally, $\kappa$ is $p$-superharmonic on $X\setminus K$ since it is the increasing limit of a sequence of $p$-superharmonic functions there (see, e.g., \cite[Lemma 3.2]{F:AAP}).\hfill$\qedhere$
		\end{itemize}
	\end{proof}

\printbibliography
\end{document}